\documentclass[a4paper,10pt]{article}
\usepackage{todonotes}
\usepackage{etex}

\usepackage{macros}

\usepackage{cite}
\usepackage{url}

\usepackage{tikz}
\usetikzlibrary{matrix}
\usetikzlibrary{arrows}
\usetikzlibrary{fit}
\usetikzlibrary{shapes}
\usepackage[all]{xy}
\usepackage{ctable}
\usepackage{amssymb}
\usepackage{amsmath}
\usepackage{fixmath}

\usepackage{multicol}

\usepackage{a4wide}
\usepackage{amsthm}

\usepackage[makeroom]{cancel}

\newtheorem{defi}{Definition}[section]
\newtheorem{remark}[defi]{Remark}

\newtheorem{thm}[defi]{Theorem}
\newtheorem{lemma}[defi]{Lemma}
\newtheorem{corollary}[defi]{Corollary}

\newtheorem{prop}[defi]{Proposition}

\newcommand{\bijar}[1][]{%
 \ar[#1]
 \ar@<0.7ex>@{}[#1]|-*=0[@]{\sim}} 

\newcommand{\wed}{\operatorname{wed}}

\title{Blocks with a generalized quaternion defect group and three simple modules over a $2$-adic ring }
\author{Florian Eisele}

\begin{document}

\maketitle

\begin{abstract}
	We show that two blocks of generalized quaternion defect with three simple modules over a sufficiently large 
	$2$-adic ring $\OO$ are Morita-equivalent if and only if the corresponding blocks over the residue field of $\OO$ are Morita-equivalent. As a corollary we show that any two blocks defined over $\OO$ with three simple modules
	and the same generalized quaternion defect group
	are derived equivalent.
\end{abstract}

\section{Introduction}

Let $(K,\OO,k)$ be a $2$-modular system.  We assume that $K$ is complete and that $k$ is algebraically closed. The aim of this article is to prove the following result:
\begin{thm}\label{thm_main}
	Assume $3 \leq n \in \N$ and $K\supseteq \Q(\zeta_{2^{n-1}}+\zeta_{2^{n-1}}^{-1})$, where $\zeta_{2^{n-1}}$ denotes a $2^{n-1}$-th root of unity.
	Let $\Lambda$ be a block of $\OO G$ and $\Gamma$ be a block of $\OO H$ for finite groups $G$ and $H$. If the defect groups of 
	$\Lambda$ and $\Gamma$ are both isomorphic to the generalized quaternion group $Q_{2^n}$ and  $\Lambda$ and $\Gamma$ both have three isomorphism classes of simple modules, then the following hold:
	\begin{enumerate}
		\item $\Lambda$ and $\Gamma$ are Morita-equivalent if and only if the $k$-algebras $k\otimes \Lambda$ and $k\otimes \Gamma$ are Morita-equivalent.
		\item $\Lambda$ and $\Gamma$ are derived equivalent.
	\end{enumerate} 
\end{thm}
The problem of classifying blocks of quaternion defect over $\OO$ arises naturally from a well-known classification result of Erdmann in \cite{TameClass}, where such blocks are classified over $k$. If the structure of these blocks is as narrowly restricted over $k$ as it is revealed to be in \cite{TameClass}, does it also follow that their structure is equally restricted over $\OO$? 
A priori it is not even clear that the number of Morita-equivalence classes of blocks over $\OO$ reducing to a single
Morita-equivalence class of blocks over $k$ is finite. However, Theorem \ref{thm_main} tells us that in our case this number is not merely finite, but in fact equal to one. In a way this was to be expected, considering the result of \cite{KessarQuat}, which proves the first part of Theorem \ref{thm_main} for one of the three possible Morita-equivalence classes of blocks of defect $Q_8$. The results of 
\cite{Olsson} concerning character values of blocks of quaternion defect and the later refinement of those results to perfect isometries between such blocks in \cite{CabanesPicaronny} also hint towards Theorem \ref{thm_main} being true, as they already show that all of the blocks over $\OO$ which are claimed to be isomorphic respectively derived equivalent in Theorem \ref{thm_main} do at the very least have isomorphic centers. Moreover, a derived equivalence between the principal block of 
$\OO \SL_2(q)$ and its Brauer correspondent has been shown to exist in \cite{GonardDiss}.

The proof of Theorem \ref{thm_main} builds mainly upon the results in \cite{TameClass} and \cite{CabanesPicaronny}.
For the case of defect $Q_8$ we also make use of the result in \cite{KessarQuat}, as it can be seen already from the 
decomposition matrix (see \eqref{decomp}) that this case is somewhat different as there are more symmetries. Here is a rough outline of the proof:
\begin{enumerate}
	\item First we prove that one of the algebras in Erdmann's classification \cite{TameClass}, namely the basic algebra of the principal block of $k\SL_2(q)$ for $q\equiv 3 \mod 4$, has a unique (or, rather, ``at most one'') 
	symmetric $\OO$-order with split semisimple $K$-span reducing to it, provided the center of the order is prescribed. Roughly speaking, 
	we do this by showing that the endomorphisms of the projective indecomposable modules 
	can be recovered as the projection of the (prescribed) center  to certain Wedderburn components. Moreover,
	the homomorphisms between different projective indecomposables viewed as bimodules over the endomorphism rings of those projective modules can also be obtained in this way. Which projections we have to take can be read off from the decomposition matrix, which we know. After that it comes down to conjugation and exploiting the symmetry of the order.
	\item The result of \cite{CabanesPicaronny} gives us a perfect isometry between any one of the blocks we are interested in and the principal block of $\OO \SL_2(q)$ for appropriately chosen $q\equiv 3 \mod 4$. In particular, this determines the center of the block in question, which is needed to apply the result of the first step.
	However, this is where a technical difficulty arises:
	above we assume that both the center of the block and the decomposition matrix are known. The $K$-span of the center is a semisimple commutative  $K$-algebra whose Wedderburn components can be associated with the rows of the decomposition matrix. In order to perform the above step properly we need to know which Wedderburn component of the center is associated with which row of the decomposition matrix. 
	The problem we are facing here boils down to the following: given a basic $\OO$-order $\Lambda$ which reduces to 
	the $k$-algebra treated in the first step as well as the basic order $\Gamma$ of the principal block of $\OO \SL_2(q)$, we have two separate isometries between the Grothendieck groups of their $K$-spans. On the one hand an isometry coming from the perfect isometry between blocks, which induces an isomorphism between the centers. On the other hand an isometry which preserves the decomposition matrix, and therefore induces an algebra isomorphism between the $K$-spans of $\Lambda$ and $\Gamma$. 
	We need to show that we can choose these two isometries equal to each other.
	The first isometry is determined up to perfect self-isometries of $\Gamma$, while 
	the second isometry is determined up to self-isometries preserving the decomposition matrix, which come from algebra automorphisms of the $K$-span of $\Gamma$. What both isometries have in common is that they map the sublattice of the Grothendieck group of $K\otimes \Lambda$ generated by the $K$-spans of projective $\Lambda$-modules onto the analogously defined sublattice of the Grothendieck group of $K\otimes \Gamma$. Therefore they differ
	by a self-isometry of the Grothendieck group of $K\otimes \Gamma$ which preserves the sublattice generated by the projective $\Gamma$-modules. Thus, in order to show that we can choose the above two isometries equal to one another, we merely have to show that the group of perfect self-isometries of $\Gamma$ and the group of 
	self-isometries of $\Gamma$ stabilizing the decomposition matrix taken together generate the group of self-isometries 
	of $\Gamma$ which preserve the lattice generated by projective modules. This is a fairly concrete problem, and proving this comes down to explicitly determining these groups of self-isometries.
	\item In the last step we generalize the result of the first two steps to all blocks of quaternion defect with three simple modules using the fact that each such block (over $k$) is derived equivalent to the principal block of $k\SL_2(q)$ for some $q\equiv 3 \mod 4$. The latter block is of course the same block we studied in the first two steps. Now we exploit the fact that a one-sided tilting complex $T$ over a $k$-algebra lifts to a tilting complex $\widehat T$ over an $\OO$-order (which we may choose freely) reducing to said $k$-algebra. The endomorphism ring of the lifted tilting complex $\widehat T$ is also reasonably well behaved, and in particular reduces to the endomorphism ring of $T$. Hence we can start with an arbitrary block (over $\OO$) of quaternion defect with three simple modules, to obtain, as the endomorphism ring of some tilting complex, an $\OO$-order reducing to the basic algebra of the principal block of $k\SL_2(q)$. This 
	$\OO$-order is then determined up to isomorphism by the first two steps. We can use this to determine the original block up to Morita-equivalence. Obviously there are a lot of technical pitfalls here that we just skimmed over in this short explanation.
\end{enumerate}
Of course there still remain some important open questions surrounding tame blocks, and blocks of quaternion defect in particular. For blocks with three simple modules one might ask whether the algebra $Q(3\mathcal B)^c$ from \cite{TameClass} actually occurs as a block of some group ring for each $c$. The answer to this, unfortunately, does not follow from our result.
Moreover, the classification of blocks of quaternion defect with two simple modules is still not entirely satisfactory, even over $k$. 

\section{The algebra $Q(3 \mathcal K)^c$} \label{section_q3k}

In this section we are going to look at a specific basic algebra of a block of defect $Q_{2^n}$ over an algebraically closed field of characteristic two, namely the algebra $Q(3 \mathcal K)^c$ from \cite{TameClass}, where $c=2^{n-2}$. This algebra is the basic algebra of the principal block of $k\SL_2(q)$ for $q\equiv 3 \mod 4$, where $q$ depends on $n$ (more about that later). The point of this section is to extract from the presentation of that algebra given in \cite{TameClass} the properties that we are going to need to prove Theorem \ref{thm_unique_lift}. Therefore a lot of what we are going to prove here will be elementary.

Let $\bar{\Lambda}$ denote the algebra $Q(3 \mathcal K)^c$ from Erdmann's classification (see the annex of \cite{TameClass}) with parameters $a=b=2$ and $c=2^{n-2}$ for some $3\leq n\in \N$. That is, $\bar{\Lambda}=kQ/I$, where $Q$ is the following quiver
\begin{equation}
Q=
\xygraph{
	!{<0cm,0cm>;<1cm,0cm>:<0cm,1cm>::}
	!{(0,1.25) }*+{\bullet_{1}}="a"
	!{(3,1.25) }*+{\bullet_{2}}="b"
	!{(1.5,-1.25) }*+{\bullet_{3}}="c"
	"b" :@/_/_{\gamma} "a"
	"a" :@/_/_{\beta} "b"
	"b" :@/_/_{\delta} "c"
	"c" :@/_/_{\eta} "b"
	"a" :@/_/_{\kappa} "c"
	"c" :@/_/_{\lambda} "a"
}
\end{equation}
and $I$ is the ideal generated by the following relations:
\begin{equation}\label{eq_def_rel}
\begin{array}{rcl}
\beta\delta&=& \kappa\lambda\kappa \\
\eta\gamma&=&\lambda\kappa\lambda \\
\delta\lambda&=&\gamma\beta\gamma \\ 
\kappa\eta&=&\beta\gamma\beta \\
\lambda\beta&=&(\eta\delta)^{2^{n-2}-1}\eta \\
\gamma\kappa&=&(\delta\eta)^{2^{n-2}-1}\delta \\
\gamma\beta\delta&=&0\\
\delta\eta\gamma&=&0\\
\lambda\kappa\eta&=&0
\end{array}
\end{equation}
\begin{prop}\label{prop_backbackforth}
	The following elements are contained in $I$:
	\begin{equation}
	\lambda\beta\gamma,\ \beta\gamma\kappa,\  \gamma\beta\delta,\ \eta\gamma\beta,\ \beta\delta\eta,\ \delta\eta\gamma,\ \eta\delta\lambda,\ \kappa\eta\delta,\ \delta\lambda\kappa,\ \lambda\kappa\eta,\ \kappa\lambda\beta,\ \gamma\kappa\lambda 
	\end{equation}
\end{prop}
\begin{proof} This is just a calculation (in each case we highlight the part of the path that we will substitute in the next step using one of the relations given in \eqref{eq_def_rel}):
	\begin{multicols}{2}
	\begin{equation}\label{eq_lbg}
	\mathbold{\lambda\beta}\gamma = (\eta\delta)^{c-1}\eta \gamma = (\eta\delta)^{c-2}\eta  \mathbold{\delta\eta\gamma} = 0
	\end{equation}
	\begin{equation}\label{eq_bgk}
	\begin{array}{rcl}
	\beta\mathbold{\gamma\kappa} &=&\beta (\delta\eta)^{c-1} \delta = \mathbold{\beta\delta} (\eta\delta)^{c-1}\\ &=&
	\kappa\lambda\kappa (\eta\delta)^{c-1} = \kappa\mathbold{\lambda\kappa\eta}\delta (\eta\delta)^{c-2}	\\&=&0
	\end{array}
	\end{equation}
	\begin{equation}
	\mathbold{\gamma\beta\delta} = 0
	\end{equation}
	\begin{equation}
	\begin{array}{rcl}
	\mathbold{\eta\gamma}\beta = \lambda\mathbold{\kappa\lambda \beta}=0 \textrm{ (uses \eqref{eq_klb}) }
	\end{array}
	\end{equation}
	\begin{equation}
	\mathbold{\beta\delta}\eta = \kappa\lambda\mathbold{\kappa\eta} = \kappa\mathbold{\lambda\beta\gamma}\beta = 0 \textrm{ (uses \eqref{eq_lbg})}
	\end{equation}
	\begin{equation}
	\mathbold{\delta \eta\gamma} = 0  
	\end{equation}
	\begin{equation}
	\eta\mathbold{\delta\lambda} = \mathbold{\eta\gamma}\beta\gamma = \lambda\kappa\mathbold{\lambda\beta\gamma} =0 \textrm{ (uses \eqref{eq_lbg})}
	\end{equation}
	\begin{equation}\label{eq_ked}
	\mathbold{\kappa\eta}\delta = \beta\mathbold{\gamma\beta\delta}=0 
	\end{equation}
	\begin{equation}
	\mathbold{\delta\lambda}\kappa = \gamma\mathbold{\beta\gamma\kappa} = 0 \textrm{ (uses \eqref{eq_bgk})}
	\end{equation}
	\begin{equation}
	\mathbold{\lambda\kappa\eta}=0
	\end{equation}
	\begin{equation}\label{eq_klb}
	\begin{array}{rcl}
	\kappa\mathbold{\lambda\beta} &=& \kappa (\eta\delta)^{c-1} \eta \\&=& \mathbold{\kappa\eta\delta} (\eta\delta)^{c-2} \eta \\&=& 0 \textrm{ (uses \eqref{eq_ked})}
	\end{array}  
	\end{equation}
	\begin{equation}
	\begin{array}{rcl}
	\mathbold{\gamma\kappa}\lambda &=& (\delta\eta)^{c-1} \mathbold{\delta\lambda}= (\delta\eta)^{c-1} \gamma\beta\gamma \\&=&
	(\delta\eta)^{c-2}\mathbold{\delta\eta \gamma}\beta\gamma =0 
	\end{array}
	\end{equation}
	\end{multicols}
\end{proof}
\begin{remark}\label{remark_dichotomy_paths}
	Think of the quiver $Q$ as a triangle and of paths in $Q$ as walks along the edges of that triangle.
	Then Proposition \ref{prop_backbackforth} tells us that going two steps forward in any direction and then going one step back in the opposite direction gives us a path which is zero in $\bar{\Lambda}$. And, in the same vein, going one step in any direction followed by two steps in the opposite direction also yields zero. This implies
	that there are two kinds of non-zero paths in $\bar \Lambda$:
	\begin{enumerate}
		\item Paths going back and forth between two vertices.
		\item Paths walking around the triangle without changing direction.
	\end{enumerate}
\end{remark}
\begin{prop}\label{prop_central_el}
	The elements $\beta\gamma+\gamma\beta$, $\delta\eta+\eta\delta$ and $\lambda\kappa+\kappa\lambda$ lie in the center of $\bar{\Lambda}$.
\end{prop}
\begin{proof}
	Proposition \ref{prop_backbackforth} implies that the product of $\beta\gamma+\gamma\beta$ with any arrow other than $\beta$ and $\gamma$ must be zero,
	regardless of whether we multiply from the left or from the right. Moreover, the following holds already in the quiver algebra $kQ$, independently of the relations in $I$:
	\begin{equation}
	\beta \cdot (\beta\gamma+\gamma\beta) = \beta\gamma\beta = (\beta\gamma+\gamma\beta)\cdot \beta \quad \textrm{and} \quad
	\gamma \cdot (\beta\gamma+\gamma\beta) = \gamma\beta\gamma = (\beta\gamma+\gamma\beta)\cdot \gamma
	\end{equation}
	The element $\beta\gamma+\gamma\beta$ also commutes with the idempotents $e_1$, $e_2$ and $e_3$, also already in $kQ$. So clearly $\beta\gamma+\gamma\beta$ is central, and the exact same reasoning implies that  $\delta\eta+\eta\delta$ and $\lambda\kappa+\kappa\lambda$ are central as well.
\end{proof}
\begin{prop}\label{prop_all_around}
	Assume $x$ is a path in $Q$. Then one of the following three possibilities holds:
	\begin{equation}
	\begin{array}{c}
	x+I = \gamma^{d_1}\cdot (\beta\gamma)^m\cdot \beta^{d_2} +I \textrm{ for some $m\in \Z_{\geq 0}$ and $d_1, d_2 \in \{0,1\}$} \\
	\textrm{or} \\
	x+I = \eta^{d_1}\cdot (\delta\eta)^m\cdot \delta^{d_2} +I \textrm{ for some $m\in \Z_{\geq 0}$ and $d_1, d_2 \in \{0,1\}$} \\
	\textrm{or} \\
	x+I = \kappa^{d_1}\cdot (\lambda\kappa)^m\cdot \lambda^{d_2} +I \textrm{ for some $m\in \Z_{\geq 0}$ and $d_1, d_2 \in \{0,1\}$}
	\end{array}
	\end{equation}
	In other words, any path in $Q$ is equivalent mod $I$ to a path going back and forth between two vertices. 
\end{prop}
\begin{proof}
	We will prove this by induction. For paths of length strictly less than two there is nothing to show. For paths of length two the claim follows from the
	first six relations given in \eqref{eq_def_rel}. Using the induction hypothesis a path of length at least three may be written as a path 
	alternating between two vertices composed with a single arrow. Either the resulting path is also alternating between two vertices, in which case there is nothing to show, or the last three arrows occurring in the path must lie in $I$ by the ``one step in one direction, two steps in the opposite direction'' criterion explained in Remark \ref{remark_dichotomy_paths}.
\end{proof}

\begin{prop}\label{prop_epi_center}
	We have
	\begin{equation}\label{eqn_sjkhdjhd}
	e_i \bar{\Lambda} e_i = e_i Z(\bar{\Lambda}) e_i \quad \textrm{ for all $i\in \{1,2,3\}$}
	\end{equation}
\end{prop}
\begin{proof}
	Proposition \ref{prop_central_el} guarantees that $e_0Z(\bar{\Lambda})e_0$ contains $\beta\gamma$ and $\kappa\lambda$, $e_1Z(\bar{\Lambda})e_1$ contains $\gamma\beta$ and $\delta\eta$
	and that $e_2Z(\bar{\Lambda})e_2$ contains $\eta\delta$ and $\lambda\kappa$. Hence in order to prove \eqref{eqn_sjkhdjhd} it suffices to show that for each $i$ the given two elements generate $e_i\bar{\Lambda}e_i$ as a $k$-algebra. But this follows directly from Proposition \ref{prop_all_around}.
\end{proof}

\begin{prop}\label{prop_cyclic_generated}
	Let $i,j\in \{1,2,3\}$ with $i\neq j$. Then  $e_i\bar{\Lambda}e_j$ is generated by a single element as a left $e_i\bar{\Lambda}e_i$-module, namely by the
	arrow in $Q$ going from $e_i$ to $e_j$.
\end{prop}
\begin{proof}
	By Proposition \ref{prop_all_around}, $e_i\bar{\Lambda}e_j$ is spanned as a $k$-vector space by paths going back and forth between the vertices $e_i$ and $e_j$, starting in $e_i$ and ending in $e_j$. But necessarily the last arrow involved in such a path is the arrow going from $e_i$ to $e_j$. Therefore the path is a product
	of a path starting and ending in $e_i$, that is, an element of $e_i\bar\Lambda e_i$, and the arrow going from $e_i$ to $e_j$. This proves that the latter arrow generates the module.
\end{proof}

\section{Generalities on the principal block of $\OO \SL_2(q)$}

In this section we gather some basic facts on the principal $2$-block of $\SL_2(q)$, in particular the case $q\equiv 3 \mod 4$. A Sylow $2$-subgroup of $\SL_2(q)$ is isomorphic to $Q_{2^n}$, where 
$n$ is the $2$-valuation of $q^2-1$. We should first remark that for every $n\geq 3$ there exists a $q$ such that a
Sylow $2$-subgroup of $\SL_2(q)$ is isomorphic to $Q_{2^n}$.
\begin{remark}
	Assume $n \geq 3$. Then the condition $q \equiv 2^{n-1}-1 \mod 2^n$ ensures that $q+1$ is divisible by two exactly $n-1$ times, and $q-1$ is divisible by two exactly once, and therefore $q^2-1=(q-1)(q+1)$ has $2$-valuation $n$. Moreover, such a $q$ satisfies $q\equiv 3 \mod 4$. By Dirichlet's theorem on arithmetic progressions there exists an infinite number of primes $q$ satisfying
	the condition $q \equiv 2^{n-1}+1 \mod 2^n$.
	
	Similarly there is a prime $q$ with $q\equiv 2^{n-1}+1 \mod 2^n$. This $q$ will of course satisfy $q\equiv 1 \mod 4$ and a Sylow $2$-subgroup of $\SL_2(q)$ is isomorphic to $Q_{2^n}$.
\end{remark}

According to \cite[Table 9.1 on page 107]{Bonnafe} the decomposition matrix of $B_0(\SL_2(q))$ for $q\equiv 3 \mod 4$ is as follows 
\begin{equation}\label{decomp}
\begin{array}{ccccc}	
& e_1 & e_2 & e_3 \\ \cmidrule[1.2pt]{2-4}
\chi_1 & 1 & 0 & 0 \\
\chi_2 & 0 & 1 & 0 \\
\chi_3 & 0 & 0 & 1 \\
\chi_4 & 1 & 1 & 1 \\
\chi_5 & 1&1&0\\
\chi_6 & 1&0&1\\\cline{2-4}
\chi_{6+r} & 0 & 1 & 1 & \ [\textrm{ exactly once for each } r=1,\ldots,l \ ] 
\end{array}
\end{equation}
where $l=2^{n-2}-1$.
By comparing this matrix with the possible decomposition matrices given in the appendix of \cite{TameClass} (or by comparing Cartan matrices), one deduces that $B_0(k\SL_2(q))$ is Morita equivalent to $Q(3\mathcal K)^c$ for $c=2^{n-2}$ (i. e. the algebra we looked at in the previous sections)
if $q\equiv 3 \mod 4$. 

\begin{prop}[Decomposition matrix]\label{prop_decomp_matrix}
	Assume $n\geq 3$.
	If $D\in \{0,1\}^{(6+l)\times 3}$ is a matrix satisfying the equation
	\begin{equation}
	D^\top \cdot D = \left( \begin{array}{ccc}
	4 & 2 & 2 \\ 2 & 2+2^{n-2} & 2^{n-2} \\ 2 & 2^{n-2} & 2+2^{n-2}
	\end{array}\right) =: C
	\end{equation}
	and if every row of $D$ is non-zero,
	then $D$ is equal to the decomposition matrix given in
	\eqref{decomp}, up to permutation of rows and columns.
\end{prop}
\begin{proof}
	For $i\in\{1,2,3\}$ denote by $c_i$ the set of all indices $j\in\{1,\ldots,6+l\}$ such that
	$D_{j,i}=1$. Since every row of $D$ is non-zero we have $|c_1\cup c_2 \cup c_3| = 6+l$.
	Moreover $|c_i|=C_{i,i}$ and $|c_i\cap c_j|=C_{i,j}$ for all $i,j\in\{1,2,3\}$.
	The inclusion-exclusion principle implies that
	\begin{equation}
	\begin{array}{rcl}
	|c_1\cap c_2 \cap c_3| &=& |c_1\cup c_2 \cup c_3| - |c_1| - |c_2| - |c_3| + |c_1\cap c_2| + |c_1\cap c_3| + |c_2\cap c_3| \\
	&=& (5+2^{n-2}) - 4 - (2+2^{n-2}) - (2+2^{n-2}) +2+2+2^{n-2} = 1
	\end{array}
	\end{equation}
	The number of rows of $D$ equal to a given row can now easily be computed. For instance
	the number of rows equal to $(1,1,1)$ is equal to $|c_1\cap c_2 \cap c_3|=1$, and the number of rows equal
	to $(1,0,0)$ is equal to $|c_1| - |c_1\cap c_2| - |c_1\cap c_3| + |c_1\cap c_2 \cap c_3|=1$.
	It is clear that the inclusion-exclusion principle determines for each vector in $\{0,1\}^{1\times 3}$ how often it occurs as a row in $D$. 
\end{proof}

\begin{prop}\label{prop_crit_decomp_01}
	Assume that $\bar \Lambda$ is a $k$-algebra, and that $e_1,\ldots,e_n\in\bar\Lambda$ is a full set of orthogonal primitive idempotents. If $e_i Z(\bar{\Lambda})=e_i\bar \Lambda e_i$ for each $i$, then every $\OO$-order $\Lambda$ in a split-semisimple $K$-algebra $A$ with $k\otimes \Lambda \iso \bar{\Lambda}$ and $\rank_\OO Z(\Lambda)=\dim_k Z(\bar{\Lambda})$ has decomposition numbers less than or equal to one. 
\end{prop}
\begin{proof}
	Since $Z(\Lambda)$ is a pure sublattice of $\Lambda$ (meaning $r\cdot x \in Z(\Lambda)$ implies $x\in Z(\Lambda)$ for any $r\in\OO$) we have an embedding $k\otimes Z(\Lambda) \hookrightarrow k\otimes \Lambda$. The image of this embedding lies in the center of $k\otimes \Lambda$, and equality of dimensions implies that 
	$Z(k\otimes \Lambda)= k\otimes Z(\Lambda)$.
	Now we have a commutative diagram
	\begin{equation}
	\xymatrix{
		Z(\Lambda) \ar@{->}[rr] \ar@{->>}[d]&& \widehat e_i \Lambda \widehat e_i \ar@{->>}[d] \\
		Z(\bar \Lambda) \ar@{->>}[rr] && e_i \bar{\Lambda} e_i
	}
	\end{equation}
	where the surjectivity of all except the upper horizontal arrow has either been shown above (for the vertical arrow on the left), holds by assumption (for the lower horizontal arrow) or is a general fact (for the rightmost vertical arrow). If we consider this as a diagram of $\OO$-modules, then it follows from the Nakayama lemma that the top arrow has to be surjective too. But that implies that $\widehat e_i \Lambda \widehat e_i$ is commutative (and hence so is $\widehat e_i A \widehat e_i$), which implies that the column in the decomposition matrix belonging to $\widehat e_i$ has entries $\leq 1$. Since this holds for all $i$, the statement is proven.
\end{proof}

Note that the previous two propositions, in conjunction with Proposition \ref{prop_epi_center}, show that any $\OO$-order reducing to $Q(3 \mathcal K)^c$ which has semisimple $K$-span and the right dimension of the center has the matrix given in \eqref{decomp} as its decomposition matrix.

\begin{prop}[Splitting field in characteristic zero]\label{prop_splitting}
	Let $\Lambda$ be a block of $\OO G$ for some finite group $G$, and assume that the defect group of $\Lambda$ is isomorphic to $Q_{2^n}$ for some $n\geq 3$. Assume moreover that $K \supseteq \Q(\zeta_{2^{n-1}}+\zeta_{2^{n-1}}^{-1})$. Then $K\otimes \Lambda$ is split.
\end{prop}
\begin{proof}
	By \cite[Corollary 31.10]{Reiner} every finite-dimensional division algebra $D$ over $\Q_2$ has a splitting field $E$ which is finite-dimensional and unramified over $Z(D)$.
	We can write $E$ as $F\cdot Z(D)$, where $F$ is an unramified extension of $\Q_2$. 
	 Since we assume the residue field of $\OO$ to be algebraically closed each unramified extension of $\Q_2$ is contained in $K$, and therefore $K\cdot Z(D) \supseteq E$. Hence $K\otimes \Lambda$ is split if and only if its center is split (which means that $K$ contains each field occurring in the Wedderburn decomposition of $Z(K\otimes \Lambda)$), which happens if and only if $K$ contains all values of all characters of the block $\Lambda$. By \cite[Proposition 4.1]{Olsson} the character values of $\Lambda$ are contained in the extension of some unramified extension of $\Q_2$ by
	$\zeta_{2^{n-1}}+\zeta_{2^{n-1}}^{-1}$, which is contained in $K$ by assumption.
\end{proof}

\section{The uniqueness of the lift of $\bar{\Lambda}$}

Theorem \ref{thm_unique_lift} below shows that the algebra $\bar \Lambda=Q(3\mathcal  K)^c$ lifts uniquely to a symmetric $\OO$-order if we prescribe the center. That is a key ingredient in proving the main theorem of this article. 

\begin{prop}\label{prop_opposite}
	Let $A$ be a finite-dimensional semisimple $K$-algebra,
	let $\Lambda\subset A$ be a symmetric $\OO$-order and let $e,f\in \Lambda$ be two idempotents with $ef=fe=0$. Then
	$$e\Lambda f = \{ x\in e A f \mid x \cdot f \Lambda e \subseteq e\Lambda e \}$$
\end{prop}
\begin{proof}
	Set $M := \{ x\in e A f \mid x \cdot f \Lambda e \subseteq e\Lambda e \}$. Clearly $e \Lambda f \subseteq M$, and it suffices to prove 
	$M \subseteq \Lambda$. 
	Let $T:\ A\times A \longrightarrow K$ be an associative $K$-bilinear form on $A$ such that $\Lambda$ is self-dual with respect to $T$
	(associativity means $T(ab,c)=T(a,bc)$, and symmetry means $T(a,b)=T(b,a)$). Such a form exists for any symmetric order.
	Then $$T(M, \Lambda) = T(eMf,\Lambda)=T(M,f\Lambda e) = T(M\cdot f\Lambda e, 1) \subseteq T(e\Lambda e, 1) \subseteq \OO$$ (where associativity and symmetry of $T$ has been used to pull $e$ and $f$ across). So 
	$M$ is contained in the dual of $\Lambda$ with respect to $T$, which is of course again $\Lambda$.	That is, $M \subseteq \Lambda$. 
\end{proof}

\begin{thm}\label{thm_unique_lift}
	Let $\bar \Lambda$ denote the algebra $Q(3 \mathcal K)^c$ from section \ref{section_q3k} with $c=2^{n-2}$ ($3\leq n \in \N$ arbitrary) and let
	\begin{equation}
	A = K\oplus K \oplus K \oplus K^{3\times 3} \oplus K^{2\times 2} \oplus K^{2\times 2} \oplus 
	\bigoplus_{i=7}^{6+l} K^{2\times 2}
	\end{equation}
	where $l=2^{n-2}-1$.
	Denote by $\eps_1,\ldots, \eps_{6+l}$ the  primitive idempotents in $Z(A)$ with the natural choice of indices.
	
	Let $\Lambda\subset A$ and $\Gamma\subset A$ be two $\OO$-orders such that all of the following holds:
	\begin{enumerate}
		\item $k\otimes \Lambda \cong k\otimes \Gamma \cong \bar\Lambda$
		\item $\Lambda$ and $\Gamma$ are symmetric
		\item $\Lambda$ and $\Gamma$ both have the decomposition matrix given in \eqref{decomp} (note: the order of the rows is fixed by the inclusion of $\Lambda$ resp. $\Gamma$ into $A$)
		\item $Z(\Lambda) = Z(\Gamma)$
	\end{enumerate}
	Then $\Lambda$ and $\Gamma$ are conjugate in $A$. 
\end{thm}

\begin{proof}
	First let us fix an isomorphism $\psi:\ k\otimes \Lambda \longrightarrow \bar{\Lambda}$. That fixes labels for the simple
	$k\otimes \Lambda$-modules, since the simple $\bar \Lambda$-modules correspond to the vertices of the quiver $Q$ from section \ref{section_q3k}, which where labeled $1$, $2$ and $3$. Without loss of generality we may choose $\psi$ in such a way that the decomposition matrix of $\Lambda$ with respect to this labeling is equal to 
	\eqref{decomp} with the order of both rows and columns being fixed. This can be done since $\bar{\Lambda}$ has sufficiently automorphisms (that is easy to check).
	Choose orthogonal primitive idempotents $\widehat e_1, \widehat e_2, \widehat e_3$ in $\Lambda$ such that 
	$\psi(1_k \otimes \widehat e_i) = e_i$ for $i\in\{1,2,3\}$. Do the analogous things for $\Gamma$ to obtain 
	$\psi', \widehat f_1, \widehat f_2$ and $\widehat f_3$. 
	The equality of the decomposition matrices implies that the systems of orthogonal idempotents $\widehat e_1,\widehat e_2,\widehat e_3$ and $\widehat f_1,\widehat f_2,\widehat f_3$ are conjugate within $A$ (as the entries of the decomposition matrix determine the ranks of those idempotents in each matrix ring summand of $A$). We can therefore replace $\Lambda$ by 
	$u_1\Lambda u_1^{-1}$ for an appropriately chosen unit $u_1\in A$, and assume that $\widehat e_i = \widehat f_i$ for 
	$i\in \{1,2,3\}$.
	
	The fact that $e_i\bar \Lambda e_i = e_i Z(\bar\Lambda) e_i$ implies that $\widehat e_i \Lambda \widehat e_i = \widehat e_i Z(\Lambda) \widehat e_i$ (and the same for $\Gamma$). To see this first note that 
	for any order there is an embedding $k\otimes Z(\Lambda) \hookrightarrow  Z(k\otimes\Lambda)$. We have 
	$\dim_k k\otimes Z(\Lambda) = \rank_\OO Z(\Lambda) = \dim_K Z(A) = 6+l$, which is equal to $\dim_k(Z(k\otimes \Lambda))=\dim_k Z(\bar{\Lambda})$, and therefore the aforementioned embedding is in fact an isomorphism.
	Since we assume that $Z(\Lambda)=Z(\Gamma)$ (as subsets of $Z(A)$), we can conclude that 
	\begin{equation}
	\widehat e_i \Lambda \widehat e_i = \widehat e_i Z(\Lambda) \widehat e_i = \widehat f_i Z(\Gamma) \widehat f_i
	= \widehat f_i \Gamma \widehat f_i
	\end{equation}
	for all $i\in\{1,2,3\}$.

	Let $i\neq j\in \{1,2,3\}$. We know that $e_i\bar\Lambda e_j$ is generated by a single element as a left $e_i \bar{\Lambda}e_i$-module. Hence the same is true 
	for $\widehat e_i\Lambda \widehat e_j$ and $\widehat e_i\Gamma\widehat e_j$ as left $\widehat e_i \Lambda \widehat e_i$-modules respectively $\widehat e_i \Gamma \widehat e_i$-modules. Let $v_{ij}$ respectively $w_{ij}$ denote generators for  $\widehat e_i\Lambda \widehat e_j$ respectively $\widehat e_i\Gamma\widehat e_j$. 
	
	Since all decomposition numbers are zero or one, it follows that the $K$-vector space  $\eps_m \widehat e_i A \widehat e_j$ 
	is at most one-dimensional for each central primitive idempotent $\eps_m \in Z(A)$ and all $i,j\in\{1,2,3\}$.
	Since  $\eps_4 \widehat e_i A \widehat e_j \iso K$ for all $i,j\in\{1,2,3\}$, and because we can replace $v_{ij}$ by $o\cdot v_{ij}$ for a unit $o\in \OO^\times$, we can ask that $\eps_4 v_{ij}$ be equal to $\pi^{d_{ij}}\cdot  \eps_4 w_{ij}$ for certain numbers $d_{ij} \in \Z$, where $\pi$ denotes a (fixed) generator of the maximal ideal of $\OO$. 
	
	Consider an element of the form
	$$
	u_2 := \sum_{i=1}^{6+l}\sum_{j=1}^3 c_{ij} \cdot \eps_i \widehat e_j
	$$
	where the $c_{ij}$ are parameters in $K^\times$ yet to be determined. Note that some of those parameters are superfluous since some of the products $\eps_i \widehat e_j$ are zero (we include those superfluous parameters only because it simplifies notation). We will try to choose the $c_{ij}$ in such a way that
	$(1-\eps_4)\cdot  u_2 v_{ij} u_2^{-1} = (1-\eps_4)\cdot w_{ij}$ for all pairs $(i,j)\in \{(1,2),(2,3),(3,1)\}$. 
	
	We will take care of the fourth Wedderburn component later on, and therefore we set 
	$c_{41}=c_{42}=c_{43}=1$. Note that for each Wedderburn-component except for the fourth there is at most one pair 
	$(i,j)\in \{(1,2),(2,3),(3,1)\}$ such that $v_{ij}$ and $w_{ij}$ projected to that component is non-zero, as all rows of the decomposition matrix \eqref{decomp} except for the fourth have at least one zero in them. 
	Seeing how $\eps_m \cdot u_2  \cdot v_{ij} \cdot u_2^{-1} = \frac{c_{mi}}{c_{mj}} \cdot \eps_m\cdot v_{ij}$ for each $m\in\{1,\ldots,6+l\}$, it follows that we can choose the $c_{mi}$ such that $\eps_m \cdot u_2 \cdot v_{ij} \cdot u_2^{-1} = \eps_m \cdot w_{ij}$ for all $4\neq m\in\{1,\ldots,6+l\}$ and for all $(i,j)\in \{(1,2),(2,3),(3,1)\}$ (where the pair $(i,j)$ is uniquely determined by $m$, as we just discussed). Without loss of generality we will replace $\Lambda$ by 
	$u_2 \Lambda u_2^{-1}$. Hence, we now have $(1-\eps_4)\cdot v_{ij}=(1-\eps_4)\cdot w_{ij}$ for all $(i,j)\in \{(1,2),(2,3),(3,1)\}$.
	
	Now consider the product $\beta\delta\lambda\in e_1\bar{\Lambda} e_1$. This is non-zero, which implies that 
	\begin{equation}
	\widehat e_1 \Lambda \widehat e_2 \Lambda \widehat e_3 \Lambda \widehat e_1 \not\subseteq \pi \cdot\Lambda
	\end{equation}
	On the other hand $v_{12}\cdot v_{23}\cdot v_{31}$ generates this (left or right) $\widehat e_1 \Lambda \widehat e_1$-module, since
	\begin{equation*}
	\begin{array}{rcl}
	\widehat e_1 \Lambda \widehat e_2 \Lambda \widehat e_3 \Lambda \widehat e_1 &=& \widehat e_1 \Lambda \widehat e_1 v_{12} \cdot \widehat e_2 \Lambda \widehat e_2 v_{23} \cdot \widehat e_3 \Lambda \widehat e_3 v_{31} \\ 
	&=& \widehat e_1 Z(\Lambda) \widehat e_1 v_{12} \cdot \widehat e_2 Z(\Lambda) \widehat e_2 v_{23} \cdot \widehat e_3 Z(\Lambda) \widehat e_3 v_{31} \\&=&Z(\Lambda) \cdot v_{12}v_{23}v_{31}
	\end{array}
	\end{equation*} 
	The analogous statement holds for $w_{12}\cdot w_{23}\cdot w_{31}$. By construction we have
	$v_{12}\cdot v_{23}\cdot v_{31} = \pi^{d_{12}+d_{23}+d_{31}}\cdot w_{12}\cdot w_{23}\cdot w_{31}$ (by looking at the decomposition matrix we see that this element has a non-zero entry only in the fourth Wedderburn component, so this follows immediately from the definition of the $d_{ij}$).  It follows that	 $d_{12}+d_{23}+d_{31}=0$, since if it were greater than zero, then $v_{12}\cdot v_{23}\cdot v_{31}$ would lie in $\pi\cdot\Lambda$, as $w_{12}\cdot w_{23}\cdot w_{31}$ already lies in 
	$\widehat e_1 \Gamma \widehat e_1 = \widehat e_1 \Lambda \widehat e_1 \subset \Lambda$. And by swapping the roles of
	$\Lambda$ and $\Gamma$ in this argument, we can also conclude that $d_{12}+d_{23}+d_{31}$ cannot be smaller than zero. 
	
	Now we can use conjugation by the element
	$$
	u_3 = (1 - \eps_4) + \eps_4\cdot (\widehat e_1 + \pi^{d_{12}}\widehat e_2 + \pi^{d_{12}+d_{23}}\widehat e_3 )
	$$
	By construction $(1-\eps_4)\cdot u_3 v_{ij}u_3^{-1}
	=(1-\eps_4)\cdot v_{ij}$ for all $i,j$, and we have already shown that $(1-\eps_4)\cdot v_{ij} = (1-\eps_4)\cdot w_{ij}$ for all $(i,j)\in\{(1,2),(2,3),(3,1)\}$. Moreover, 
	$\eps_4\cdot  u_3 v_{12}u_3^{-1} = \eps_4 w_{12}$, 
	$\eps_4\cdot  u_3 v_{23}u_3^{-1} = \eps_4 w_{23}$
	and $\eps_4\cdot  u_3 v_{31}u_3^{-1} = \pi^{d_{12}+d_{23}+d_{31}}\cdot \eps_4 w_{31}= \eps_4 w_{31}$. Hence we can conclude that $u_3\cdot v_{ij}\cdot u_3^{-1}=w_{ij}$ for all $(i,j)\in\{(1,2),(2,3),(3,1)\}$.
	Replace $\Lambda$ by $u_3\Lambda u_3^{-1}$.
	
	Now we have $\widehat e_i\Lambda \widehat e_i = \widehat e_i\Gamma \widehat e_i$ for all $i\in\{1,2,3\}$ and 
	$\widehat e_i\Lambda \widehat e_j = \widehat e_i\Gamma \widehat e_j $ for all $(i,j)\in\{(1,2),(2,3),(3,1)\}$.
	By the assumption that $\Lambda$ and $\Gamma$ are symmetric, and using Proposition \ref{prop_opposite} it follows that
	$\widehat e_i\Lambda \widehat e_j = \widehat e_i\Gamma \widehat e_j $ for all $(i,j)\in\{(2,1),(3,2),(1,3)\}$ as well.
	Therefore we get $\Lambda=\Gamma$, which finishes the proof.
\end{proof}

We should remark at this point that the condition ``$Z(\Lambda) = Z(\Gamma)$'' in the previous theorem is quite strong, 
and certainly stronger than merely asking that the centers should be isomorphic.
Hence, in order to apply Theorem \ref{thm_unique_lift} to blocks of quaternion defect, we must first study to which extent the perfect isometries between them provide us with information on the embedding of the center into the Wedderburn-decomposition of the block. That is what the next section is about.

\section{Perfect self-isometries of $B_0(\OO \SL_2(q))$ for $q\equiv 3 \mod 4$}

In this section we will study self-isometries preserving projectivity and perfect self-isometries of $B_0(\OO \SL_2(q))$ for $q\equiv 3 \mod 4$. 
For the most part we restrict our attention to the case $n>3$ ($n$ being the $2$-valuation of $q^2-1$), but in the case $n=3$ the main 
result of this section, Corollary \ref{corollary_iso_factorization}, is just the same as
\cite[Proposition 1.1]{KessarQuat}. 
Since later on we will have to deal with orders which are not a priori known to be blocks of groups rings, we are going to formulate 
Proposition \ref{prop_self_iso} below in a slightly more general setting than needed for the purposes of this section.

\begin{prop}[Isometries preserving projectivity]\label{prop_self_iso}
	Let $\Lambda$ be an $\OO$-order with split-semisimple 
	$K$-span and decomposition matrix as given in \eqref{decomp}. 
	We index the elements 
	of $\Irr_K(K\otimes \Lambda)$ by the numbers $\{1,\ldots,6+l\}$ using the same ordering as for the rows of the decomposition matrix \eqref{decomp}. We write self-isometries of the Grothendieck group $K_0(K\otimes\Lambda)$ as signed permutations, using the following convention for a signed cycle: $(i_1,\ldots, i_k)$ denotes the permutation that sends $|i_j|$ to $i_{j+1}$ for each $j\in \Z/k\Z$. 
	
	Denote the projective indecomposable $\Lambda$-modules 
	belonging to the first, second and third column of \eqref{decomp} by $P_1$, $P_2$ and $P_3$, respectively.
	Let $\varphi:\ K_0(K\otimes \Lambda)\longrightarrow K_0(K\otimes \Lambda)$ be an isometry that 
	maps the $\Z$-lattice $\langle [K\otimes P] \mid \textrm{$P$ projective $\Lambda$-module} \rangle_{\Z}$ onto itself. Then 
	\begin{equation} \label{eqn_generators_selfiso}
		\varphi \in \langle \pm \id, (2, 3)(5, 6),\ (-1, -4)(2, 3)(-5)(-6),\ (2, 4)(-1, -3)(-6)  \rangle \cdot {\rm Sym}(\{7,\ldots,6+l\})
	\end{equation}
\end{prop}

\begin{proof}
	Clearly $\varphi$ induces a permutation of $\pm\Irr(K\otimes \Lambda)$.
	If $[V]\in \Irr(K\otimes \Lambda)$, then $(\varphi ([P_i]), [V]) = ([P_i], \varphi^{-1}([V])) \in \{-1,0,1\}$, where we made use of the fact that all decomposition numbers of
	$\Lambda$ are $\leq 1$. 
	For each $[P_i]$ there is a $[V_i]\in \Irr(K\otimes \Lambda)$ which occurs with multiplicity $1$ in $[P_i]$ and with multiplicity $0$ in each $[P_j]$ with $i\neq j$.
	If we write $\varphi([P_i])=a_{i,1}[P_1]+a_{i,2}[P_2]+a_{i,3}[P_3]$, then $a_{i,j}=([V_j], \varphi([P_i)])=(\varphi^{-1}([V_j]), [P_i])\in \{-1,0,1\}$. In the same vein, for each $j\neq j' \in \{1,2,3\}$ we have a $[V]\in \Irr(K\otimes \Lambda)$ such that $([P_l],[V])=1$ if and only if $l\in \{j,j'\}$ and
	$([P_l],[V])=0$ for the unique $l\in \{1,2,3\}-\{j,j'\}$.
	Then $a_{i,j}+a_{i,j'}=([V], \varphi([P_i]))=(\varphi^{-1}([V]), [P_i])\in\{-1,0,1\}$.
	Hence the only possibilities for $\varphi([P_i])$ are 
	linear combinations of $[P_1]$, $[P_2]$ and $[P_3]$ with coefficients $0$, $1$ and $-1$, where either only one coefficient is non-zero, or one coefficient is equal to zero, one is equal to $+1$ and one is equal to $-1$.
	Let $[P] = a [P_1]+b[P_2]+c[P_3]$. Then
	\begin{equation}
	\begin{array}{rcl}
	([P],[P]) &=& \left(\begin{array}{ccc}a&b&c\end{array}\right)\cdot \left( \begin{array}{ccc}
	4 & 2 & 2 \\ 2 & 2+2^{n-2} & 2^{n-2} \\ 2 & 2^{n-2} & 2+2^{n-2}
	\end{array}\right)\cdot \left(\begin{array}{c}a \\ b\\ c\end{array}\right)
	\\ \\
	&=& 4a^2 + (2+2^{n-2})\cdot (b^2+c^2) + 4ab + 4ac + 2^{n-1}bc
	\end{array}
	\end{equation}
	By symmetry we only have to check the cases $(a,b,c)=(1,-1,0)$ and $(a,b,c)=(0,1,-1)$. We get 
	$([P], [P])= 2+2^{n-2}$ respectively $([P], [P])= 4$. As $(\varphi([P_i]), \varphi([P_i]))=([P_i],[P_i])$ is equal to $4$ if $i=1$ and equal to $2+2^{n-2}$ if $i\in\{1,2\}$ it follows that $\pm([P_1]-[P_{2}])$ and $\pm([P_1]-[P_3])$ are suitable images for both $[P_2]$ and $[P_3]$, and $\pm([P_2]-[P_3])$ is a suitable image for $[P_1]$.
	
	Now we list all possible images of the triple $[P_1]$, $[P_2]$ and $[P_3]$ which stabilize the Cartan matrix, which 
	means that they may be induced by an isometry $\varphi$, but we will still have to find a signed permutation on $\pm 
	\Irr_K(K\otimes \Lambda)$ that induces them. We give such a permutation in each case. 
	\begin{equation}\label{eqn_all_elements}
	\begin{array}{llll}
		\varphi([P_1]) & \varphi([P_2]) & \varphi([P_3]) & \textrm{Signed permutation on $\{1,\ldots,6+l\}$}  \\
		  \cmidrule[1.2pt]{1-4}
		\pm [P_1] & \pm[P_2] & \pm[P_3] & \pm \id \\
		& \pm[P_3] & \pm[P_2] & \pm(2,3)(5,6) \\
 		& \pm ([P_1]-[P_2]) & \pm([P_1]-[P_3]) & \pm (1,4)(-2)(-3)(5,6)(-\id_{\{7,\ldots,6+l\}})\\
		& \pm ([P_1]-[P_3]) & \pm([P_1]-[P_2]) & \pm (1,4)(-2,-3) (-\id_{\{7,\ldots,6+l\}})\\
		\pm ([P_2]-[P_3]) & \pm [P_2] & \mp ([P_1]-[P_2]) & \pm (2,4)(-1,-3)(-6) \\
		& \mp[P_3] & \pm ([P_1]-[P_3]) & \pm (1,2,-4,-3)(5,-6)(-\id_{\{7,\ldots,6+l \}}) \\
		& \mp ([P_1]-[P_2]) & \pm [P_2 ] & \pm (-1,-3,4,2)(5,-6)\\
		& \pm ([P_1]-[P_3]) & \mp[P_3] & \pm (1,2)(-3,-4)(-6)(-\id_{\{ 7,\ldots,6+l \}})
	\end{array}
	\end{equation}
	Note that the given signed permutations are unique up to composition with permutations which fix $[P_1]$, $[P_2]$ and $[P_3]$. That are precisely all the permutations that fix $\{1,\ldots, 6\}$ point-wise. This proves the lemma (it is easy to verify that the elements given in \eqref{eqn_generators_selfiso} are generators for the group we just determined).
\end{proof}

Now we are interested in the question which of these self-isometries are actually perfect if we pick $\Lambda = B_0(\OO \SL_2(q))$.
For the rest of the section we will fix some prime $q$ with $q \equiv 3 \mod 4$ and assume that $K\otimes B_0(\OO\SL_2(q))$ is split, which is equivalent to asking that $K\supseteq \Q(\zeta_{2^{n-1}}+\zeta_{2^{n-1}}^{-1})$ ($n$ is the $2$-valuation of $q^2-1$).

\begin{remark}[Character table]\label{remark_chartab}
	We are going to need the character table of $\SL_2(q)$, which can be found in
	\cite[Table 5.4]{Bonnafe}, to check perfectness of characters. The table below is the part we are interested in, namely the characters which lie in 
	the principal $2$-block of $\SL_2(q)$. We already replaced 
	those character values for which there is a 
	more explicit description due to our constraints on $q$ (concretely: $q_0=-q$, $\theta_0(\eps)=1$, $\alpha_0(\eps)=\eps$).
	\begin{equation}\label{chartab}
	\arraycolsep=4pt\def\arraystretch{1.3}
	\begin{array}{llcccc}
	&&\eps I_2 & \operatorname{d}(a) & \operatorname{d}'(\xi) & \eps u_\tau \\
	&& \eps \in\{\pm 1\} & a \in \mu_{q-1} & \xi \in \mu_{q+1} & \eps, \tau \in \{ \pm 1 \}\\
	&|C_{\SL_2(q)}(g)||& q(q^2-1) & q-1 & q+1 & 2q \\ 
	&o(g) & o(\eps) & o(a) & o(\xi) & q\cdot o(\eps)\\
	\cmidrule[1.2pt]{1-6} 
	1&\chi_1 & 1 & 1 & 1 & 1\\ 
	R'_+(\theta_0)&\chi_2&\frac{1}{2}(q-1)&0&-\theta_{{0}} \left( \xi \right) &\frac{1}{2}(-1+\tau \sqrt {-q}) \\ R'_-(\theta_0)&\chi_3&\frac{1}{2}(q-1)&0&-\theta_{{0}} \left( \xi \right) &\frac{1}{2}(-1-\tau\sqrt {-q})\\ 
	\operatorname{St}&\chi_4&q&1&-1&0 \\ 
	R_+(\alpha_0)&\chi_5&\frac{1}{2} \left( q+1 \right) \eps&\alpha_{{0}}
	\left( a \right) &0&\frac{1}{2}\eps\left( 1+\tau\sqrt {-q}
	\right) \\ 
	R_-(\alpha_0)&\chi_6&\frac{1}{2}\left( q+1 \right) \eps&\alpha
	_{{0}} \left( a \right) &0&\frac{1}{2}\eps\left( 1-\tau\sqrt {-q}
	\right) \\\hline
	R'(\theta) & \chi_{6+i} &  (q-1)\cdot\theta(\eps) & 0 & -\theta(\xi)-\theta(\xi)^{-1} & -\theta(\eps) 
	\end {array} 
	\end{equation}
	The top row and the leftmost column contains the names for the conjugacy classes respectively characters used 
	in \cite{Bonnafe}. The symbol $\mu_j$ denotes the group of $j$-th roots of unity in the algebraic closure 
	of $\F_q$ (i. e., a cyclic group of order $j$ if $\gcd(q,j)=1$). The symbol $\theta_0$ denotes the unique ordinary character of order two of the group $\mu_{q+1}$, and the symbol $\alpha_0$ denotes the unique ordinary character of order two of 
	the group $\mu_{q-1}$. The parameter $\theta$ in the last row ranges over all characters of order $2^i$ of the group 
	$\mu_{q+1}$ for $i \geq 2$, although different $\theta$ may still yield the same character $R'(\theta)$. Namely, we have $R'(\theta_1) = R'(\theta_2)$ for $\theta_1\neq \theta_2$ if and only if $\theta_1$ is the complex conjugate of $\theta_2$. 
	
	We should recall that, due to our choice of $q$, the $2$-valuation of $q-1$ is one, and the $2$-valuation of
	$q+1$ is $n-1$. In particular $\theta$ ranges over $(2^{n-1} - 2)/2 = 2^{n-2}-1$ different values. The following 
	will be useful later on:
	\begin{enumerate}
		\item Since $q-1$ has $2$-valuation one, the $2$-valuation of the order of an element $a\in \mu_{q-1}$ is either $0$ or $1$. We have
		\begin{equation}
		\alpha_0(a) = \left\{ \begin{array}{ll} -1 & \textrm{if $a$ has even order}\\ 1 & \textrm{if $a$ has odd order}\end{array} \right.
		\end{equation}
		\item Since $q+1$ has $2$-valuation $n-1$, the $2$-valuation of the order of an element $\xi\in \mu_{q+1}$
		is at most $n-1$. We have
		\begin{equation}\label{eqn_sjhjhdj}
		\theta_0(\xi) = \left\{ \begin{array}{ll} -1 & \textrm{if the $2$-valuation of $o(\xi)$ is $n-1$}\\ 1 & \textrm{otherwise}\end{array} \right.
		\end{equation}
	\end{enumerate}
\end{remark}

\begin{prop}\label{prop_self_perfect_iso}
	\begin{enumerate}
	\item The involution $g\mapsto g^{-1}$ on $\OO\SL_2(q)$  induces a self-isometry $(2,3)(4,5)$ of the principal block. This self-isometry is also induced by an automorphism (see Proposition  \ref{prop_gl2_action} below).
	\item Alvis-Curtis duality (see \cite[Chapter 8.4]{Bonnafe}) swaps the trivial character and the Steinberg character, which correspond to the 
	first and fourth row of the decomposition matrix \eqref{decomp}. By inspecting the rows of \eqref{eqn_all_elements}
	we see that the signed permutation it induces on irreducible characters is either in $(1, 4)(-2)(-3)(5, 6)(-\id_{\{7,...,6+l\}})\cdot \operatorname{Sym}(\{7,\ldots,6+l\})$
	or in $(1, 4)(-2, -3)(-\id_{\{7,...,6+l\}})\cdot \operatorname{Sym}(\{7,\ldots,6+l\})$.
	\end{enumerate}
\end{prop}

Note that the existence of the above perfect self-isometries implies that $Z(B_0(\OO\SL_2(q)))$ has automorphisms which induce the corresponding (unsigned) permutations on the Wedderburn components (or, equivalently, the primitive idempotents of $Z(K\otimes B_0(\OO\SL_2(q)))$).

It is now quite natural to ask whether the self-isometry inducing the permutation $(2,4)(-1,-3)(-6)$ is also perfect, since then all three generating (signed) permutations in equation \eqref{eqn_generators_selfiso} would come from perfect isometries. It turns out that this is false. But, as we will see in a moment, we can find an element $\sigma\in \operatorname{Sym}(\{7,\ldots,6+l\})$ such that $(2,4)(-1,-3)(-6) \circ \sigma$ is perfect. Nevertheless, we will start by looking at the permutation $(2,4)(-1,-3)(-6)$, and the corresponding character of
$G\times G$:
\begin{equation}
	\begin{array}{rcl}
	\iota_0(g,h) &:=& -\chi_1(g) \chi_3(h)-\chi_3(g)\chi_1(h)+\chi_2(g)\chi_4(h) + \chi_4(g)\chi_2(h) + \chi_5(g)\chi_5(h)
	\\ &&- \chi_6(g)\chi_6(h)
	+ \sum_{i=7}^{6+l} \chi_i(g) \chi_i(h)
	\end{array}
\end{equation}
Since perfect characters form an additive group, and we know that the character of $G\times G$ representing the
identity permutation is perfect, we may just as well study the character
\begin{equation}\label{eqn_def_mu}
		\begin{array}{rcl}
		\mu(g,h) &=& \left(\sum_{i=1}^{6+l} \chi_i(g)\chi_i(h)\right) - \iota_0(g,h) \\ 
		&=&\chi_1(g)\chi_1(h) + \chi_3(g)\chi_3(h) +\chi_1(g)\chi_3(h) + \chi_3(g)\chi_1(h) +
		\chi_2(g)\chi_2(h) \\ &&+ \chi_4(g)\chi_4(h) - \chi_2(g)\chi_4(h)-\chi_4(g)\chi_2(h) + 2\chi_6(g)\chi_6(h)
		\end{array}
\end{equation}
Now we will look at how exactly $\mu(g,h)$ fails to be perfect, and how we can rectify that. Unfortunately, checking that a character is perfect is usually  a somewhat tedious computation, and what follows is no exception. 

\begin{lemma}\label{lemma_mu_almost_perf}
	The character $\mu(g,h)$ defined in \eqref{eqn_def_mu} has the following properties:
	\begin{enumerate}
		\item If $g$ has even order, and $h$ has odd order, then 
			$\mu(g,h)=\mu(h,g)=0$.
		\item If $g$ and $h$ both have odd order, then
			\begin{equation}
				\frac{\mu(g,h)}{|C_{\SL_2(q)}(g)|} \in \OO \quad\textrm{and}\quad \frac{\mu(g,h)}{|C_{\SL_2(q)}(h)|}\in\OO
			\end{equation}
		\item If $g$ and $h$ both have even order, then
			\begin{equation}
					\frac{\mu(g,h)}{|C_{\SL_2(q)}(g)|} \in \OO \quad\textrm{and}\quad \frac{\mu(g,h)}{|C_{\SL_2(q)}(h)|}\in\OO
			\end{equation}
			except possibly if $g$ and $h$ belong to the conjugacy classes labeled by $\operatorname{d}'(\xi_1)$ and 
			$\operatorname{d}'(\xi_2)$
			in the character table \eqref{chartab}, and $\xi_1$ and $\xi_2$ both have order divisible by
			$2^{n-1}$.
	\end{enumerate}
\end{lemma}
\begin{proof}
	Note that only the characters $\chi_1$ through $\chi_6$ are involved in $\mu$, and  the entries in the
	rows of the character table \eqref{chartab} depend only slightly on the concrete parameters $\eps$, $a$, $\xi$, $\eps$ and $\tau$ that specify the actual 
	conjugacy class (note: there are in fact two different parameters called ``$\eps$''). In fact, we only need to know $\eps$ for the first column, $\alpha_0(a)$ for the second, $\theta_0(\xi)$ for the third and $\eps$ as well as $\tau$ for the fourth. If  we are looking at the character values of an element of odd order, then $\eps=1$, $\alpha_0(a)=1$, $\theta_0(\xi)=1$ and $\eps=1$ respectively (only $\tau$ still depends on the actual conjugacy class).   If we are looking at the character values of an element of even order, then $\eps=-1$, $\alpha_0(a)=-1$, $\theta_0(\xi)$  still depends on the order of $\xi$ and $\eps=-1$  (again, $\tau$ depends on the actual conjugacy class). Here we used the last part of Remark \ref{remark_chartab}.
	Hence we can write the values of $\mu(g,h)$ for $g\in \{\eps_1 I_2, \operatorname{d}(a_1), \operatorname{d}'(\xi_1), \eps_1u_{\tau_1}\}$  of odd order and $h\in \{\eps_2 I_2, \operatorname{d}(a_2), \operatorname{d}'(\xi_2), \eps_2u_{\tau_2}\}$  of even order into a $4\times 4$-matrix, whose entries will
	only depend on $\tau_1$, $\theta_0(\xi_2)$, and $\tau_2$. One can compute this
	$4\times 4$-matrix by hand or using a computer (which seems more sensible seeing how this is a rather lengthy computation), and one obtains the zero matrix, which proves the first part of our assertion.
	
	In the same vein we can put the values of $\frac{\mu(g,h)}{|C_{\SL_2(q)}(g)|}$ for $g\in \{\eps_1 I_2, \operatorname{d}(a_1), \operatorname{d}'(\xi_1), \eps_1u_{\tau_1}\}$   and $h\in \{\eps_2 I_2, \operatorname{d}(a_2), \operatorname{d}'(\xi_2), \eps_2u_{\tau_2}\}$  both of odd order into a $4\times 4$-matrix.
	Note that $|C_{\SL_2(q)}(h)|$ and $|C_{\SL_2(q)}(g)|$ and do not depend on $\eps_1$, $a_1$, etc. at all.
	We obtain the following matrix:
	\begin{equation}
			 \arraycolsep=4pt\def\arraystretch{1.5}
		  \begin {array}{c|cccc} 
		  & \eps_1 I_2 & \operatorname{d}(a_1) & \operatorname{d}'(\xi_1) & \eps_1 u_{\tau_1} \\ \hline 
		  \eps_2 I_2 &{\frac {q+1}{ \left( q-1 \right) q}}&{
		 	\frac {2}{ \left( q-1 \right) q}}&0&-{\frac {-1+\tau_{{2}}\sqrt {-q}}{
		 		\left( q-1 \right) q}}\\
		  \operatorname{d}(a_2) & 2{\frac {q+1}{q-1}}&\frac{4}{
		 q-1}&0&-2{\frac {-1+\tau_{{2}}\sqrt {-q}}{q-1}}
		 \\ 
		 \operatorname{d}'(\xi_2)&0&0&0&0\\ 
		 \eps_2 u_{\tau_2}&-{\frac {(\tau_1\sqrt{-q} - 1)(1+q)}{2q}}&-{\frac {-1+
		 			\tau_{{1}}\sqrt {-q}}{q}}&0&{\frac {(\tau_1\sqrt{-q} - 1)(\tau_2\sqrt{-q} - 1)
		 			}{2q}}\end {array} 	
	\end{equation}
	To see that all of the entries lie in $\OO$ it suffices to know the following: $2$ divides $q-1$ and $q+1$,
	$q-1$ divides both $2$ and $q+1$, and both $2$ and $q-1$ divide $1\pm \sqrt{-q}$. The latter is owed to the fact that $\frac{1\pm \sqrt{-q}}{2}$ is integral due to our choice of $q$. Since $\mu(g,h)=\mu(h,g)$ for all $g$ and $h$ the proof of the second part of our assertion is complete.
	
	To finish the proof let us look at the values of $\frac{\mu(g,h)}{|C_{\SL_2(q)}(g)|}$ for $g\in \{\eps_1 I_2, \operatorname{d}(a_1), \operatorname{d}'(\xi_1), \eps_1u_{\tau_1}\}$   and $h\in \{\eps_2 I_2, \operatorname{d}(a_2), \operatorname{d}'(\xi_2), \eps_2u_{\tau_2}\}$  both of even order. We get almost the same matrix as above:
	\begin{equation}\label{eqn_sjhjhdkjdhk}
	\arraycolsep=4pt\def\arraystretch{1.5}
	\begin {array}{c|cccc} 
	& \eps_1 I_2 & \operatorname{d}(a_1) & \operatorname{d}'(\xi_1) & \eps_1 u_{\tau_1} \\ \hline 
	\eps_2 I_2 &{\frac {q+1}{ \left( q-1 \right) q}}&{
		\frac {2}{ \left( q-1 \right) q}}&0&-{\frac {-1+\tau_{{2}}\sqrt {-q}}{
			\left( q-1 \right) q}}\\
	\operatorname{d}(a_2) & 2{\frac {q+1}{q-1}}&\frac{4}{
		q-1}&0&-2{\frac {-1+\tau_{{2}}\sqrt {-q}}{q-1}}
	\\ 
	\operatorname{d}'(\xi_2)&0&0&
	\frac{2 (\theta_0(\xi_1) \theta_0(\xi_2)-\theta_0(\xi_1)-\theta_0(\xi_2)+1)}{q+1}
	&0\\ 
	\eps_2 u_{\tau_2}&-{\frac {(\tau_1\sqrt{-q} - 1)(1+q)}{2q}}&-{\frac {-1+
			\tau_{{1}}\sqrt {-q}}{q}}&0&{\frac {(\tau_1\sqrt{-q} - 1)(\tau_2\sqrt{-q} - 1)
		}{2q}}\end {array} 	
	\end{equation}	
	The only entry that is different is the $(3,3)$-entry, which will in general not lie in $\OO$. However, this entry is zero unless both
	$\theta_0(\xi_1)$ and $\theta_0(\xi_2)$ are equal to $-1$, which, according to \eqref{eqn_sjhjhdj}, happens only
	if $\xi_1$ and $\xi_2$ both have order divisible by $2^{n-1}$. That finishes the proof.
\end{proof}

\begin{lemma}\label{lemma_new_selfiso}
	Let
	\begin{equation}
		\iota_1(g,h) := \sum_{\theta} (R'(\theta)(g) R'(\theta)(h) - R'(\theta)(g) R'(\theta\cdot \theta_0)(h))
	\end{equation}
	using the notation for the irreducible characters from \eqref{chartab}. The summation index $\theta$
	ranges over the same $2^{n-2}-1$ characters of $\mu_{q+1}$ as in the character table.
	
	Then $\iota_0 - \iota_1$ is a perfect isometry, and clearly the induced signed permutation on irreducible 
	characters is $(2,4)(-1,-3)(-6)\circ\sigma$ for an element $\sigma\in\operatorname{Sym}(\{7,\ldots, 6+l\})$ of
	order two.
\end{lemma}
\begin{proof}
	It is clear by definition that $\iota_0-\iota_1$ is an isometry, we only need to check that it is perfect.
	To do that we may just as well prove that $\mu + \iota_1$ is perfect. 
	
	Let us first note that 
	$\iota_1(g,h) \neq 0$ if and only if both $g$ and $h$ have order divisible by $2^{n-1}$. It is clear that 
	if $h$ has order not divisible by $2^{n-1}$, then $\iota_1(g, h)=0$ because $R(\theta)(h) = R(\theta\cdot \theta_0)(h)$ which means that each individual summand in the definition of $\iota_1$ is zero. The fact that
	$\iota_1(g,h)$ is zero whenever $g$ has order not divisible by $2^{n-1}$ can be seen by rearranging the sum.
	To be precise, we apply an index shift to obtain the following equality:
	\begin{equation}
		\begin{array}{rcl}
		\iota_1(g,h) &=&  \sum_{\theta} R'(\theta)(g) R'(\theta)(h) - \sum_{\theta}R'(\theta)(g) R'(\theta\cdot \theta_0)(h) \\ \\
			&=&  \sum_{\theta} (R'(\theta)(g) R'(\theta)(h) - \sum_{\theta} R'(\theta \cdot \theta_0)(g) R'(\theta)(h)) \\ \\
			&=& \iota_1(h,g) 
		\end{array}
	\end{equation}
	We conclude that $\mu + \iota_1$ satisfies all of the conditions 1 through 3 from the statement of Lemma 
	\ref{lemma_mu_almost_perf} just as $\mu$ does (because the relevant values of $\mu + \iota_1$ are actually equal to 
	those of $\mu$). 
	
	It remains to prove that if $g$ and $h$ both have order divisible by $2^{n-1}$ (which we will assume from here on out), then
	\begin{equation}
	\frac{\mu(g,h) +\iota_1(g,h)}{|C_{\SL_2(q)}(g)|}=\frac{\mu(g,h) +\iota_1(g,h)}{q+1} \stackrel{!}{\in} \OO
	\end{equation}
	 To show this we can simply evaluate $\mu$ and $\iota_1$: $\mu(g,h)=8$ (from \eqref{eqn_sjhjhdkjdhk}), 
	and since $\theta_0(h)=-1$ the formula for $\iota_1(g,h)$ simplifies to the following:
	\begin{equation}\label{eqn_sjhjddd}
		\iota_1(g,h) = 2\cdot \sum_{\theta} (\theta(g)+\theta(g)^{-1})(\theta(h)+\theta(h)^{-1})
	\end{equation}
	The latter is technically abuse of notation, since $\theta$ is not really defined on $g$ (or $h$, for that matter), but rather on $\xi\in\mu_{q+1}$, where $\operatorname{d}'(\xi)$ is the ``standard'' representative (from \eqref{chartab}) of the conjugacy class that $g$ (or $h$) is an element of.
	
	Recall that if we take all characters of $\mu_{q+1}$ of two-power order except those of order $\leq 2$, and partition these characters into sets of cardinality two containing a character and its complex conjugate, then 
	the summation index $\theta$ in \eqref{eqn_sjhjddd} ranges over representatives of those sets. Since the values of $R'(\theta)$ and $R'(\bar{\theta})$ are actually equal, we might as well let $\theta$ range over all characters of two-power order $> 2$, and then divide the resulting sum by two. Using the fact that the characters of $\mu_{q+1}$ of two-power order are obtained by 
	mapping $g$ respectively $h$ to all possible $2^{n-1}$-st roots of unity, we get the following:
	\begin{equation}\label{eqn_hghegke}
		\begin{array}{rcl}
		\iota_1(g,h) &=& \sum_{i=2}^{n-1} \sum_{\sigma \in \Gal(\Q(\zeta_{2^i})/\Q)} (\zeta_{2^i}^\sigma + 
		(\zeta_{2^i}^\sigma)^{-1}) ((\zeta_{2^i}^z)^\sigma + ((\zeta_{2^i}^z)^\sigma)^{-1}) \\ \\
		&=& \sum_{i=2}^{n-1} \sum_{\sigma \in \Gal(\Q(\zeta_{2^i})/\Q)} 
		(\zeta_{2^i}^{1+z} + \zeta_{2^i}^{1-z} + \zeta_{2^i}^{-1+z} + \zeta_{2^i}^{-1-z})^\sigma \\ \\
		&=& \ldots \textrm{ (continued below)}
		\end{array}
	\end{equation}
	Here $\zeta_{2^i}$ denotes a primitive $2^i$-th root of unity, and $z\in \Z$ is chosen such that
	$\theta(h)=\theta(g^z)$ for all $\theta$ of two-power order (i.e. the two-part of $h$ is conjugate to the 
	$2$-part of $g^z$; in particular, $z$ is odd). Note that for any $j \in \Z$ we have
	\begin{equation}\label{eqn_kldkldjlk}
		\sum_{i=0}^{n-1} \sum_{\sigma \in \Gal(\Q(\zeta_{2^i})/\Q)} (\zeta_{2^i}^j)^\sigma =\left\{\begin{array}{ll} 
			0 & \textrm{if $j \not\equiv 0 \mod 2^{n-1}$} \\
			2^{n-1} & \textrm{if $j\equiv 0 \mod 2^{n-1}$}
		\end{array}\right.
	\end{equation}
	This follows for instance from column orthogonality in the character table of the cyclic group of order $2^{n-1}$. 
	This allows us to simplify \eqref{eqn_hghegke} as follows (note: we need to pay attention to the different ranges for the summation index $i$ in \eqref{eqn_hghegke} and \eqref{eqn_kldkldjlk}):
	\begin{equation}
		\begin{array}{rcl}
			\ldots &=& N(z) - \sum_{i=0}^1 \sum_{\sigma \in \Gal(\Q(\zeta_{2^i})/\Q)} 
			(\zeta_{2^i}^{1+z} + \zeta_{2^i}^{1-z} + \zeta_{2^i}^{-1+z} + \zeta_{2^i}^{-1-z})^\sigma \\ \\
			&\stackrel{z\textrm{ odd}}{=}& N(z) - 8
		\end{array}
	\end{equation}
	Here $N(z)$ denotes $2^{n-1}\cdot (\delta_{1+z,0} + \delta_{1-z,0} + \delta_{-1+z,0} + \delta_{-1-z,0})$ (where $\delta_{a,b}=1$ if $a\equiv b \mod 2^{n-1}$ and $\delta_{a,b}=0$ otherwise). We can now conclude that
	\begin{equation}
		\mu(g,h) + \iota_1(g,h) = N(z) 
	\end{equation}
	which is divisible by $2^{n-1}$, and therefore also by $q+1$ (in $\OO$).
\end{proof}

\begin{corollary}\label{corollary_iso_factorization}
	Let 
	\begin{equation}
	\varphi:\ K_0(K\otimes B_0(\OO\SL_2(q)))\longrightarrow K_0(K\otimes B_0(\OO\SL_2(q)))
	\end{equation}
	be an isometry that 
	maps the $\Z$-lattice $\langle [K\otimes P] \mid \textrm{$P$ projective $B_0(\OO \SL_2(q))$-module} \rangle_{\Z}$ onto itself. Then $\varphi$ can be written as $\varphi_1 \circ \varphi_2$, where 
	$\varphi_1\in \operatorname{Sym}(\{7,\ldots, 6+l\})$ (we identify isometries and signed permutations as in Proposition \ref{prop_self_iso}) and $\varphi_2$ is a perfect isometry. 
\end{corollary}
\begin{proof}
	For $n=3$ the assertion has been proved in \cite[Proposition 1.1]{KessarQuat}. 
	For $n>3$ the claim follows from Proposition \ref{prop_self_iso}, which yields a factorization of $\varphi$ as
	$\varphi_1\circ \varphi_2$ with $\varphi_1\in\operatorname{Sym}(\{7,\ldots, 6+l\})$. Proposition \ref{prop_self_perfect_iso}
	and Lemma \ref{lemma_new_selfiso} then show that $\varphi_2$ composed with an appropriate element of 
	$\operatorname{Sym}(\{7,\ldots, 6+l\})$ is a perfect isometry.
\end{proof}

\section{Fixing the Wedderburn embedding}\label{section_fix_wedderburn}
Throughout this section we will assume the following:

\begin{enumerate}
\item $\Lambda$ is a symmetric $\OO$-order with split semisimple $K$-span.
\item $k\otimes \Lambda$ is basic.
\item The decomposition matrix of $\Lambda$ is the one given in \eqref{decomp}, up to permutation of rows and columns.
\item We fix a prime $q\equiv 3 \mod 4$ and an isomorphism (which we assume to exist)
\begin{equation}
	\Phi:\ Z(\Gamma) \stackrel{\sim}{\longrightarrow} Z(\Lambda)
\end{equation}  
	where $\Gamma$ is the basic order of $B_0(\OO \SL_2(q))$. We assume moreover that there is an isometry 
	\begin{equation}
		\widehat \Phi: K_0(K\otimes \Gamma) \longrightarrow K_0(K\otimes \Lambda)
	\end{equation}
	with the following properties:
	\begin{enumerate}
	\item if $\eps \in Z(K\otimes \Gamma)$ is a central primitive idempotent, $V$ is the simple $K\otimes 
	\Gamma$-module associated with $\eps$, and $W$ is the simple $K\otimes \Lambda$-module associated with the central 
	primitive idempotent $(\id_K\otimes \Phi)(\eps) \in Z(K\otimes \Lambda)$, then $\widehat \Phi([V])=\pm[W]$.
	\item $\widehat \Phi(\langle[K\otimes P] \mid \textrm{$P$ is a projective $\Gamma$-module}\rangle_\Z) =
	\langle[K\otimes Q] \mid \textrm{$Q$ is a projective $\Lambda$-module}\rangle_\Z$
	\end{enumerate}
\end{enumerate}
Note that since $\Gamma$ is the basic order of $B_0(\OO \SL_2(q))$ we may identify $K_0(K\otimes \Gamma)$ with $K_0(K\otimes B_0(\OO\SL_2(q)))$.
The following diagram visualizes the situation we are looking at:
\begin{equation}
	\xymatrix@C=-.2\textwidth{
		Z(\Gamma) \bijar[d]_{\Phi}="p1" \ar@{^{(}->}[dr]^-{\wed_\Gamma|_{Z(\Gamma)}} \\ 
		Z(\Lambda) \ar@{^{(}->}[r]_-{\wed_\Lambda|_{Z(\Lambda)}} 
		\ar@{^{(}->}[d] & **[r] Z(A) = K\oplus K\oplus K\oplus K\oplus K\oplus K\oplus \overbrace{K\oplus\ldots \oplus K}^{\textrm{$l$ copies}} \ar@{^{(}->}[d] \\
		\Lambda \ar@{^{(}->}[r]_{\wed_\Lambda} & **[r] A = K\oplus K \oplus K \oplus K^{3\times 3} \oplus K^{2\times 2} \oplus K^{2\times 2} \oplus \underbrace{K^{2\times 2} \oplus \ldots \oplus K^{2\times 2}}_{\textrm{$l$ copies}} 	
		\POS"2,2"\ar@{}|-{\xcancel\circlearrowleft}"p1"
		\POS"2,1"\ar@{}|-{\circlearrowleft}"3,2"			
	}
\end{equation}
where the maps $\wed_\Lambda: \Lambda \longrightarrow A$ and $\wed_\Gamma: \Gamma \longrightarrow A$ denote Wedderburn embeddings whose images have the decomposition matrix \eqref{decomp} with the order of the rows being fixed. The potential non-commutativity of the top part of this diagram is the main obstacle in showing there is only one Morita equivalence class of quaternion blocks over $\OO$ reducing to the Morita equivalence class of $B_0(k\SL_2(q))$. We will now investigate this non-commutativity in greater detail.

\begin{remark}\label{remark_isometry_center}
		\begin{enumerate}
		\item Let us denote the primitive idempotents of $Z(A)$ by $\eps_1, \ldots, \eps_{6+l}$.
		A $K$-algebra automorphism of $Z(A)$ is given by a permutation of these idempotents, i. e. each automorphism is 
		of the form $\alpha_\sigma$ for some $\sigma \in \operatorname{Sym}(\{1,\ldots, 6+l\})$, where
		\begin{equation}
		\alpha_\sigma(\eps_i)=\eps_{\sigma(i)}
		\end{equation} 
		Since $A$ is split the permutation $\sigma$ determines the automorphism uniquely. 
		
		\item\label{item_det} A self-isometry of $K_0(A)$ determines an automorphism of $Z(A)$. To be more specific, we can construct the 
		corresponding permutation $\sigma$ by forgetting the signs in the signed permutation acting on $\pm\Irr(A)$.
		Similarly, an isometry between $K_0(K\otimes \Gamma)$ and $K_0(A)$ determines an embedding $Z(\Gamma) \longrightarrow Z(A)$.
		
		\item Two Wedderburn embeddings of a 
		commutative $\OO$-order into $Z(A)$ differ only by an automorphism of $Z(A)$.
		\end{enumerate}
\end{remark}

\begin{prop}\label{prop_wed_emb_perf_iso}
	There is a permutation $\sigma \in \operatorname{Sym}(\{7,\ldots, 6+l\})$ such that
	\begin{equation}
		\wed_\Lambda(Z(\Lambda)) = \alpha_{\sigma}(\wed_\Gamma (Z(\Gamma)))
	\end{equation}
	\emph{Remark}: Note that $\alpha_\sigma$ extends to an automorphism of $A$, and $\alpha_\sigma^{-1} \circ \wed_{\Lambda}$ can be 
	regarded as another Wedderburn embedding of $\Lambda$, whose image has the same decomposition matrix 
	as the image of $\wed_\Lambda$ (where the order of the rows is fixed).
\end{prop}
\begin{proof}
	We may (canonically) identify $K_0(A)$ with $\mathbb Z^{6+l}$ (equipped with the usual euclidean scalar product), since the 
	Wedderburn components of $A$ are ordered. With this identification the isomorphisms
	$\id_K\otimes \wed_\Lambda:\ K\otimes \Lambda \longrightarrow A$ and $\id_K\otimes \wed_\Gamma:\ K\otimes \Gamma 
	\longrightarrow A$ come from isometries  
	\begin{equation}
	\widehat \wed_\Lambda: K_0(K\otimes \Lambda)\longrightarrow \Z^{6+l}\quad \textrm{respectively} \quad\widehat \wed_\Gamma: K_0(K\otimes \Gamma)\longrightarrow \Z^{6+l}
	\end{equation}
	 that send the equivalence classes of simple $K\otimes \Lambda$- respectively $K\otimes \Gamma$-modules to the standard basis of $\Z^{6+l}$. Due to our assumption on the decomposition matrices of the images of $\wed_\Lambda$ and $\wed_\Gamma$, these isometries send the $\Z$-lattice 
	$\langle [P] \mid \textrm{$P$ projective $\Lambda$-module} \rangle_{\Z}$ respectively $\langle [P] \mid \textrm{$P$ projective $\Gamma$-module} \rangle_{\Z}$ onto the $\Z$-sublattice of $\mathbb Z^{6+l}$ generated by the columns 
	of the decomposition matrix \eqref{decomp}.
	
	Now let us compare the maps $\wed_\Gamma|_{Z(\Gamma)}$ and $\wed_\Lambda|_{Z(\Lambda)}\circ \Phi$. They are induced by the isometries $\widehat \wed_\Gamma$ and $\widehat \wed_\Lambda\circ \widehat\Phi$, both of which send $\langle[K\otimes P] \mid \textrm{$P$ a projective $\Gamma$-module} \rangle_{\mathbb Z}$ 
	to the $\Z$-sublattice of $\Z^{6+l}$ which is generated by the columns of the decomposition matrix (here we use the assumptions we made on $\widehat \Phi$). By Corollary \ref{corollary_iso_factorization} we can conclude that 
	$\widehat \wed_\Gamma$ and $\widehat \wed_\Lambda\circ \widehat\Phi$ differ by a self-isometry of 
	$K_0(K\otimes \Gamma)$ which is the composition of a perfect self-isometry of $\Gamma$ and 
	an element of $\operatorname{Sym}(\{7,\ldots,6+l\})$. 
	
	Assume $\gamma$ is an automorphism of $Z(\Gamma)$ which comes from a self-isometry $\widehat\gamma$ of $K_0(K\otimes \Gamma)$.
	Then $\wed_\Gamma  \circ \gamma$ comes from the isometry $\widehat \wed_\Gamma \circ 
	\widehat \gamma$., In light of the previous paragraph, we can choose a $\widehat \gamma = \widehat \gamma_1 \circ \widehat \gamma_2$, with $\widehat\gamma_1 \in \operatorname{Sym}(\{7,\ldots, 6+l\})$ and $\widehat\gamma_2$ being a perfect isometry, such that $\widehat \wed_\Lambda \circ \widehat 
	\Phi  = \widehat \wed_\Gamma \circ \widehat \gamma$.  Now we can pull $\widehat \gamma_1$ through $\widehat \wed_\Gamma$, to obtain
	$\widehat\gamma_1' \circ \widehat \wed_\Gamma \circ \widehat \gamma_2$ for a certain self-isometry $\widehat \gamma_1' $ of $K_0(A)$. Note that there is an automorphism of $Z(A)$ induced by the self-isometry 
	$\widehat\gamma_1'$, and this automorphism is equal to $\alpha_	\sigma$ with $\sigma \in \cdot \operatorname{Sym(\{7,\ldots,6+l\})}$. By Remark \ref{remark_isometry_center} (\ref{item_det}) we can conclude that
	\begin{equation}
		\wed_\Lambda|_{Z(\Lambda)} \circ \Phi = \alpha_\sigma \circ \wed_\Gamma|_{Z(\Gamma)} \circ  \gamma_2
	\end{equation} 
	where $\gamma_2$ is the automorphism of $Z(\Gamma)$ induced by the perfect isometry $\widehat \gamma_2$.
	Now we can simply take the images of the maps on both sides of this equation, and our claim immediately follows.
\end{proof}

\begin{prop}\label{prop_gl2_action}
	Assume $q\equiv \pm 3 \mod 4$. 
	There is a non-trivial outer automorphism of $\SL_2(q)$ induced by conjugation with an appropriately chosen 
	element of $\GL_2(q)$. This outer automorphism induces a non-trivial permutation of 
	the simple $B_0(k\SL_2(q))$-modules.
\end{prop}
\begin{proof}
	By looking at the parametrization of the conjugacy classes of $\SL_2(q)$ given in \cite{Bonnafe}, we see that there are two  conjugacy classes of unipotent matrices in $\SL_2(q)$, while there is only one such class in $\GL_2(q)$. Hence 
	$\GL_2(q)$ acts non-trivially on the conjugacy classes of $\SL_2(q)$, and therefore also on its (absolutely) irreducible characters. It can be seen by inspection of the character table \cite[Table 5.4]{Bonnafe} (see also \cite[Excercise 4.3]{Bonnafe}) that the unique non-trivial outer automorphism of $\SL_2(q)$ induced by an element of $\GL_2(q)$ swaps the characters $R_{+}(\alpha_0)$ and $R_{-}(\alpha_0)$ as well as 
	the characters $R_+'(\theta_0)$ and $R_-'(\theta_0)$. The reduction of the characters  $R_+'(\theta_0)$
	and $R_-'(\theta_0)$ to $2$-regular conjugacy classes gives the Brauer characters belonging to the simple 
	$B_0(k\SL_2(q))$-modules $\overline{\operatorname{St}}_+^k$ and $\overline{\operatorname{St}}_-^k$ (this is proven in \cite[Section 9.4.4]{Bonnafe}). This shows that twisting by  $\alpha$ swaps the two simple modules $\overline{\operatorname{St}}_+^k$ and $\overline{\operatorname{St}}_-^k$.
\end{proof}

Note that an automorphism of $B_0(\OO \SL_2(q))$ gives rise to an automorphism of the basic order of 
$B_0(\OO \SL_2(q))$ inducing the same permutation on isomorphism classes of simple modules. Therefore we get the following:

\begin{corollary}\label{corollary_full_permutation_auto}
	Assume $q\equiv \pm 3 \mod 4$. By $n$ we denote the $2$-valuation of the order of $\SL_2(q)$, and  we assume $n\geq 3$. 
	Again, let $\Gamma$ be the basic algebra of $B_0(\OO \SL_2(q))$ and let $S_1$, $S_2$ and $S_3$ denote its simple modules.
	Then for every automorphism $\alpha$ of $k\otimes \Gamma$ there exists an automorphism $\widehat \alpha$ of 
	$\Gamma$ such that $(\id_k\otimes \widehat \alpha)\circ \alpha^{-1}$ fixes all simple modules, that is,
	$S_i^{(\id_k\otimes \widehat \alpha)\circ \alpha^{-1}} \iso S_i$ for each $i\in \{1,2,3\}$ 
	
	\emph{Note:} $S_i^{(\id_k\otimes \widehat \alpha)\circ \alpha^{-1}}$ denotes the module obtained from $S_i$ by letting $k\otimes \Gamma$ act on it through the automorphism ${(\id_k\otimes \widehat \alpha)\circ \alpha^{-1}}$. Later on we will also use the analogous notation for bimodules.
\end{corollary}
\begin{proof}
	We know that the Cartan matrix of $B_0(\SL_2(q))$ looks as follows:
	\begin{equation}
		\underbrace{\left(\begin{array}{ccc} 4 & 2 & 2 \\ 2& 2+ 2^{n-2} & 2^{n-2} \\ 2 & 2^{n-2} & 2+2^{n-2} \end{array}\right)}_{\textrm{if $q\equiv 3 \mod 4$}} \quad \textrm{or} \quad
		\underbrace{\left(\begin{array}{ccc} 2^n & 2^{n-1} & 2^{n-1} \\ 2^{n-1}& 2+ 2^{n-2} & 2^{n-2} \\ 2^{n-1} & 2^{n-2} & 2+2^{n-2} \end{array}\right)}_{\textrm{if $q\equiv 1 \mod 4$}}
	\end{equation}
	If $\alpha$ is an automorphism of $k\otimes \Gamma$, then the dimension of the endomorphism ring of the projective cover of $S_i^\alpha$ is the same as the dimension of the endomorphism ring of the projective cover of $S_i$ (for each $i$). If either $n>3$ or $q\equiv 1 \mod 4$, then the diagonal entries of the above Cartan matrices are not all equal, which implies that $\alpha$ needs to fix one isomorphism class of simple modules, and it might swap the other two. Hence, if $\alpha$ is non-trivial, then it necessarily needs to induce the same permutation on simple modules as the automorphism of $\Gamma$ coming from Proposition \ref{prop_gl2_action}. 
	
	The case $q\equiv 3 \mod 4$ and $n=3$ is special, since then the Cartan matrix imposes no restriction 
	on the permutation of the simple modules induced by $\alpha$. However, in that case, \cite[Theorem A]{KessarQuat}
	implies that $\Gamma \iso \OO \tilde A_4 = \OO Q_8 \rtimes C_3$, and \cite[Lemma 1.2]{KessarQuat} implies that
	this $\OO$-order has an automorphism which induces a permutation of order three on isomorphism classes of simple modules. It follows that this automorphism together with the automorphism from Proposition \ref{prop_gl2_action}
	generates a full symmetric group on three points, which implies our assertion.
\end{proof}

\section{Transfer to other algebras of quaternion type}

In this section we are going to use a technique reminiscent of the one used in \cite{EiseleDerEq} to get a theorem similar to Theorem \ref{thm_unique_lift} for arbitrary blocks of quaternion defect. Note that, technically, Theorem \ref{thm_unique_lift} cannot be applied to any block of quaternion type yet, but only to their basic algebras if they happen to be isomorphic 
to $Q(3\mathcal K)^c$. The fact that Theorem \ref{thm_unique_lift} remains valid if one  replaces the algebra $\bar{\Lambda}=Q(3\mathcal K)^c$ by an algebra Morita-equivalent to it would be a side-note at best.
The main idea in \cite{EiseleDerEq} was that, up to technicalities, one can in fact replace $\bar{\Lambda}=Q(3\mathcal K)^c$ by an algebra derived equivalent to it, instead of just Morita-equivalent. In our case the
technical side of this argument is in fact much simpler than in \cite{EiseleDerEq}, and hence we will only have to use well-known facts about derived equivalences.

\begin{defi}[Admissible lifts of quaternion blocks]\label{defi_admissible}
	We call an $\OO$-order $\Lambda$  \emph{admissible} if the following three conditions hold:
	\begin{enumerate}
		\item $K\otimes \Lambda$ is split semisimple
		\item $\Lambda$ is symmetric 
		\item For some prime $q$ there is an isometry $\widehat \Phi:\ K_0(K\otimes \Lambda) \longrightarrow K_0(K\otimes B_0(\OO\SL_2(q)))$
		and an isomorphism $\Phi:\ Z(\Lambda) \longrightarrow Z(B_0(\OO \SL_2(q)))$
		which satisfy the assumptions made at the beginning of section \ref{section_fix_wedderburn}.
	\end{enumerate}	
	If  $\bar{\Lambda}$ is a finite-dimensional $k$-algebra, then an \emph{admissible lift} is an admissible $\OO$-order 
	$\Lambda$ with $k\otimes \Lambda \iso \bar{\Lambda}$ 
\end{defi}

\begin{lemma}[Admissible lifts of $Q(3 \mathcal K)^c$]\label{lemma_q3k}
Let $\Lambda$ be an admissible lift of $Q(3 \mathcal K)^c$ ($c=2^{n-2}$, $3\leq n$ arbitrary). Assume that 
$K\otimes B_0(\OO \SL_2(q))$ is split, where $q\equiv 3 \mod 4$ is a prime such that $q+1$ has $2$-valuation $n-1$.
Then $\Lambda$ is isomorphic to the basic order of $B_0(\OO\SL_2(q))$. 
\end{lemma}
\begin{proof}
	Let $\Gamma$ be the basic order of $B_0(\OO\SL_2(q))$. We need to
	show that $\Lambda$ and $\Gamma$ are isomorphic.
	Proposition \ref{prop_epi_center} and Proposition \ref{prop_crit_decomp_01} imply that both $\Lambda$ and $\Gamma$ have decomposition numbers $\leq 1$. Now Proposition \ref{prop_decomp_matrix} implies that the decomposition matrices of $\Lambda$ and $\Gamma$ are both equal to the one given in \eqref{decomp}, up to permutation of rows and columns. Let 
	\begin{equation}
		A := K \oplus K \oplus K \oplus K^{3\times 3} \oplus K^{2\times 2} \oplus K^{2\otimes 2} \oplus \bigoplus_{i=7}^{6+l} K^{2\times 2}
	\end{equation}
	From the decomposition matrices of $\Lambda$ and $\Gamma$ we know that this algebra $A$ is isomorphic to both $K\otimes\Lambda$ and $K\otimes\Gamma$. We can choose Wedderburn embeddings $\wed_\Lambda:\ \Lambda \longrightarrow A$ and $\wed_\Gamma:\ \Gamma \longrightarrow A$ such that the decomposition matrices of their images are both equal
	to \eqref{decomp} (the order of the rows now being fixed).
	 Moreover, Proposition \ref{prop_wed_emb_perf_iso} implies that there is an automorphism $\alpha$ of $A$ permuting the Wedderburn components $7,\ldots, 6+l$
	such that $\alpha(\wed_\Lambda (Z(\Lambda)))= \wed_\Gamma(Z(\Gamma))$. Note that $\wed_\Lambda (Z(\Lambda)) = Z(\wed_\Lambda(\Lambda))$ (and the same for $\Gamma$). Now we may apply Theorem \ref{thm_unique_lift} to
	the $\OO$-orders $\alpha(\wed_\Lambda (\Lambda))$ and $\wed_\Gamma(\Gamma)$.
	It follows that $\alpha(\wed_{\Lambda}(\Lambda))$ and $\wed_\Gamma(\Gamma)$ are conjugate in $A$, which implies that $\Lambda$ and $\Gamma$ are isomorphic.
\end{proof}

At this point we have to look at the other two Morita equivalence classes of $2$-blocks with defect group 
$Q_{2^n}$. These are the algebras $Q(3\mathcal A)_2^c$ and $Q(3\mathcal B)^c$ from the appendix of \cite{TameClass} (note: we parametrize these algebras as in \cite{HolmDerEq}, using only a single parameter ``$c$'', as that is the only undetermined parameter in the context of blocks).
The article \cite{HolmDerEq} gives us an explicit (one-sided) two-term tilting complex in 
$\mathcal K^b(\projC - Q(3\mathcal A)_2^c)$ with endomorphism ring $Q(3\mathcal B)^c$
and an explicit two-term tilting complex in $\mathcal K^b(\projC - Q(3\mathcal A)_2^c)$ with 
endomorphism ring $Q(3\mathcal K)^c$ (the algebra we have been looking at exclusively so far).

We will need a few well-known results on derived equivalences in order to get a version of Lemma 
\ref{lemma_q3k} for the algebras $Q(3\mathcal A)_2^c$ and $Q(3\mathcal B)^c$. The first one is a theorem of Rickard which tells us that a derived equivalences between two $k$-algebras give rise to derived equivalences between two $\OO$-orders reducing to these respective $k$-algebras. The caveat of this is that only one of the two $\OO$-orders
can be chosen freely, while the other one is then determined up to isomorphism by this choice.
Note that while our notation is mostly standard, there is one peculiarity that may be worth pointing out: we consider one-sided tilting complexes as complexes of right modules, and the endomorphism ring has the usual composition as its multiplication (that is, the endomorphism ring of a module acts on the module from the left). With this convention, 
a ring is derived equivalent to the endomorphism ring of a tilting complex, rather than the opposite ring thereof.
\begin{thm}[see {\cite[Theorem 3.3]{RickardLiftTilting}}]\label{thm_rickard_lifting}
	If $\Lambda$ is an $\OO$-order and $T\in \mathcal K^b(\projC-k\otimes \Lambda)$ is a 
	tilting complex, then there is a tilting complex $\widehat T\in \mathcal K^b(\projC-\Lambda)$ (unique up to isomorphism) with $k\otimes \widehat T \iso T$. Moreover, $\End_{\mathcal D^b(\Lambda)}(\widehat T)$
	is an $\OO$-order, and $k\otimes \End_{\mathcal D^b(\Lambda)}(\widehat T) \iso \End_{\mathcal D^b(k\otimes \Lambda)}(T)$.
\end{thm}
A second fact we will need is that two-term tilting complexes are determined by their terms.
\begin{thm}[{see \cite[Corollary 8]{JensenXuZDegenerations}}]\label{thm_jxz}
	Let $A$ be a $k$-algebra and let $P_1$ and $P_0$ be projective $A$-modules. Then there is at most one 
	tilting complex of the form $0\longrightarrow P_1\longrightarrow P_0\longrightarrow 0$, up to isomorphism
	in $\mathcal K^b(\projC-A)$.
\end{thm}
On top of that we are going to use some facts from \cite{DerCatDerEq}, \cite{RouquierZimmermann} and \cite{DerEq}, which are formulated for an
$R$-algebra $A$ which is projective as an $R$-module, where $R$ is an arbitrary commutative ring. 
These facts can therefore be applied to finite dimensional $k$-algebras and $\OO$-orders equally.
Namely, if we have a one-sided
tilting complex $T\in \mathcal K^b(\projC-A)$ whose endomorphism ring is isomorphic to $B$, then there is a two-sided tilting complex $X\in \mathcal D^b(\modC-B^{\op} \otimes_R A)$ whose restriction to $A$ is isomorphic to $T$ in the derived category (see \cite[Corollary 3.5]{DerCatDerEq}).
Moreover, by \cite[Lemma 2.2]{RouquierZimmermann} such an $X$ can be chosen in such a way that the restriction of each term of $X$ to both $A$ and $B$ is projective (without necessarily being projective as a $B$-$A$-bimodule). An equivalence between $\mathcal D^b(B)$ and $\mathcal D^b(A)$ is then afforded by the functor $-\otimes_{B}^{\mathbb L} X$. If we choose $X$ is such a way that all its terms are projective as $B$-modules, then we may replace the left derived tensor product by the ordinary tensor product of complexes. 
We denote the inverse of $X$ by $X^\vee$. We have $X\otimes_A^{\mathbb L} X^\vee \iso B$ in $\mathcal D^b(B^{\op} \otimes_R B)$ and $X^\vee \otimes_B^{\mathbb L} X \iso A$ in $\mathcal D^b(A^{\op} \otimes_R A)$. For symmetric algebras $A$ the complex $X^\vee$ can be computed as $\Hom_R(X, R)$ (see \cite[Section 9.2.2]{DerEq}). 

In the case of self-injective algebras it is fairly easy to check whether a derived equivalence is actually a Morita-equivalence. A one-sided tilting complex $T$ over a self-injective algebra $A$ is always isomorphic 
(in $\mathcal D^b(A)$) to a tilting complex whose highest and lowest degree non-zero terms
are in the same degree as its highest and lowest degree non-zero homologies (as both epimorphisms onto projectives and embeddings of projectives split in this case). In particular, a 
one-sided tilting complex which has non-zero homology only in a single degree is isomorphic to the stalk complex
of a projective module (which has to be a progenerator), and its endomorphism ring in $\mathcal D^b(A)$ is isomorphic to the endomorphism ring of that projective module. It follows that $A$ is Morita equivalent to $\End_{\mathcal D^b(A)}(T)$. If there is a two-sided tilting complex for two algebras which has homology concentrated in a single degree, then these algebras are Morita-equivalent because the restriction of said tilting complex to either side is
isomorphic to a one-sided tilting complex, and isomorphisms in $\mathcal D^b(A)$ preserve homology (pretty much by definition). 

We should also note that if $A$ is a symmetric $k$-algebra, and $B$ is the endomorphism ring of a two-term tilting complex over $A$, then $A$ is also the endomorphism ring of a two-term tilting complex over $B$. This follows simply from the fact that
$X^\vee$ can be computed as $\Hom_k(X, k)$, and the fact that we can choose a one-sided tilting complex in such a way that its non-zero terms are concentrated between the highest and lowest degree non-zero homology.

If we have a one-sided tilting complex $T$ over a symmetric $\OO$-order $\Lambda$, then its endomorphism ring $\Gamma$ is an $\OO$-order by \cite{TiltedSymmIsSymm}, and by \cite[Theorem 2.1 and Corollary 2.2]{DerCatDerEq} the fact that $\Gamma$ is an $\OO$-order implies that $k\otimes T$ is a tilting complex over $k\otimes \Lambda$ with endomorphism ring
$k\otimes \Gamma$. If $k\otimes T$ has homology concentrated in a single degree then it is isomorphic to the stalk complex of a progenerator $P$ in $\modC-k\otimes \Lambda$. We know that there is a progenerator $\widehat P$ of $\modC-\Lambda$ with
$k\otimes \widehat P \iso P$. Hence the stalk complex associated with $\widehat P$ is a tilting complex over $\Lambda$ reducing to the stalk complex associated with $P$. By Theorem \ref{thm_rickard_lifting} such a complex is unique
up to (quasi-)isomorphism, and therefore $T$ must be isomorphic to the stalk complex associated with $\widehat P$, which means that $\Gamma$ is Morita-equivalent to $\Lambda$. 

Assuming $A$ is a symmetric $R$-algebra we can also give an explicit description of two-sided tilting complexes with homology concentrated in a single degree.
Namely, such a complex $X$ is quasi-isomorphic to the stalk-complex of its non-zero homology, which we will denote by $M$. Since we also know that its restriction to either side is quasi-isomorphic to a one-sided tilting complex, which under the assumptions made is quasi-isomorphic to the stalk complex of a progenerator, it follows that $M$ is projective as a left and as a right module. Of course we can do the same for $X^\vee$, which must be quasi-isomorphic to some $A$-$B$-bimodule $M^\vee$, also projective from the left and from the right. Due to projectivity it follows that we do not have to bother with the derived tensor product (as discussed above), and we can conclude that $M\otimes_A M^\vee \iso B$ as a $B$-$B$-bimodule and $M^\vee \otimes_B M \iso A$ as an $A$-$A$-bimodule. That is, $M$ is an invertible bimodule. 
In the case where $A$ is equal to $B$ and $A$ is a basic $k$-algebra or a basic $\OO$-order we can go even further:
in that case $M\iso A^\alpha$ for some automorphism $\alpha$ of $A$.

\begin{prop}\label{prop_tilted_admissible}
	Let $\Lambda$ be an admissible $\OO$-order, and let $T\in \mathcal K^b(\projC-\Lambda)$ be a tilting complex.
	Then $\Gamma := \End_{\mathcal D^b(\Lambda)}(T)$ is an admissible $\OO$-order.
\end{prop}
\begin{proof}
	Since $\Lambda$ is symmetric, $\Gamma$ is both an $\OO$-order and symmetric (see \cite{TiltedSymmIsSymm}).
	Moreover, $K\otimes T$ is a tilting complex over $K\otimes \Lambda$ with endomorphism ring $K\otimes \Gamma$ (this is elementary), and since $K\otimes \Lambda$ is assumed to be
	split semisimple, so is $K\otimes \Gamma$ (since two semisimple algebras are derived equivalent if and only if they are Morita equivalent). 
	
	It remains to find an isometry $\widehat \Psi:\ K_0(K\otimes \Gamma) \longrightarrow K_0(K\otimes \Lambda)$ 
	mapping the sublattice of $K_0(K\otimes \Gamma)$ generated by $K$-spans of projective $\Gamma$-modules
	onto the analogously defined sublattice of $K_0(K\otimes \Lambda)$. Moreover there should be an isomorphism $\Psi:\ Z(\Gamma) \longrightarrow Z(\Lambda)$ whose $K$-linear extension maps a primitive idempotent $\eps_V$ in $Z(K\otimes \Gamma)$ to $\eps_W$, the primitive idempotent in $Z(K\otimes \Lambda)$ that belongs to 
	$[W] = \pm \widehat \Psi ([V])$. Once we have found these maps, it is immediate that the composition
	of $\widehat \Psi$ respectively $\Psi$ with the maps $\widehat \Phi$ respectively $\Phi$ from the definition of admissibility (of $\Lambda$) yields the maps needed for $\Gamma$ to be admissible.
	First let us choose a two-sided tilting complex $X$ with terms that are projective
	as left $\Gamma$-modules and as right $\Lambda$-modules such that $-\otimes_{\Gamma} X$ affords an equivalence between $\mathcal D^b(\modC-\Gamma)$ and $\mathcal D^b(\modC-\Lambda)$. We choose $\widehat \Psi$ to be the induced map from 
	$K_0(K\otimes \Gamma)$ to $K_0(K\otimes \Lambda)$, that is,  $[V]\mapsto [V\otimes _{K\otimes \Gamma} (K \otimes X)]$. This map $\widehat \Psi$ is an isometry since $K\otimes X$ is a two-sided tilting complex, and it maps the sublattice of $K_0(K\otimes \Lambda)$ spanned by 
	the projective modules into the sublattice of $K_0(K\otimes \Gamma)$ spanned by projectives. The reason for the latter is simply that if 
	$P$ is projective, then all terms of $P\otimes_\Gamma X$ are projective, and the homomorphism between 
	the Grothendieck groups is defined by applying $-\otimes_\Gamma X$ and then taking the alternating sum 
	of the terms. Since the same argument applies to $-\otimes_{K\otimes \Lambda}(K\otimes X^\vee)$, it also follows that $\widehat \Psi$ maps the sublattice generated by projective $\Gamma$-modules surjectively onto the sublattice generated by projective $\Lambda$-modules.
	
	To get the required isomorphism between the centers we first note that it is well known that the centers of derived equivalent algebras are isomorphic. 
	Concretely, an isomorphism $\Psi: Z(\Gamma)\longrightarrow Z(\Lambda)$ can be obtained by identifying both $Z(\Gamma)$ and $Z(\Lambda)$ with the ring of
	endomorphisms of $X$ (in $\mathcal D^b(\Lambda\otimes \Gamma^{\op})$). That means in particular that $x \cdot \Psi(z)=z\cdot x$ for all $x\in H^i(X)$ ($i$ arbitrary) and all $z \in Z(\Gamma)$. 
	This implies that if an idempotent $\eps \in Z(K\otimes \Gamma)$ acts non-trivially on 
	a simple $K\otimes \Gamma$-module $V$, then the image of $\eps$ under the $K$-linear extension of $\Psi$ 
	acts non-trivially on $V\otimes_{K\otimes \Gamma} (K\otimes X)$, which by definition becomes $\pm\widehat \Psi([V])$ in the Grothendieck group. This shows that $\widehat \Psi$ and $\Psi$ have the required properties.
\end{proof}

\begin{prop}\label{prop_lift_transfer}
	Assume that $\bar \Lambda$ is a $k$-algebra and $\Lambda_0$  is an admissible lift of $\bar{\Lambda}$ such that the following hold: 
	\begin{enumerate}
	\item If $\Lambda$ is an arbitrary  admissible lift 
	of $\bar \Lambda$, then $\Lambda\iso \Lambda_0$.
	\item Every automorphism of $K_0(k\otimes \Lambda_0)$ which is induced by some element of 
	$\Aut_k(k\otimes \Lambda_0)$ is also induced by an element of $\Aut_\OO(\Lambda_0)$. 
	\end{enumerate}
	Then the following holds for every basic $k$-algebra $\bar \Lambda'$  which is derived equivalent to $\bar \Lambda$ 
	by means of a two-term tilting complex:
	There is an admissible lift $\Lambda_0'$ of $\bar \Lambda'$ such that  if $\Lambda'$ is an arbitrary  admissible lift 
	of $\bar \Lambda'$, then $\Lambda'\iso \Lambda_0'$.
\end{prop}
\begin{proof}
	Let $T$ be a two-term tilting complex over $\bar \Lambda'$ with $\End_{\mathcal D^b(\bar \Lambda')}(T)\iso \bar \Lambda$. Then, for any admissible lift $\Lambda'$ of $\bar \Lambda'$ there exists a tilting complex $\widehat T\in \mathcal K^b(\projC-\Lambda)$ such that $k\otimes \widehat T \iso T^{\varphi}$, where $\varphi: k\otimes \Lambda'
	\stackrel{\sim}{\longrightarrow} \bar \Lambda'$ is an isomorphism. Define 
	$\Lambda$ to be $\End_{\mathcal D^b(\Lambda')}(\widehat T)$, and let $X$ denote a two-sided tilting complex in
	$\mathcal D^b(\Lambda^{\op} \otimes \Lambda')$ whose restriction to 
	the right is isomorphic to $\widehat T$. We assume without loss that all terms of $X$ are projective as left $\Lambda$-modules and as right $\Lambda'$-modules.
	
	We let $\Lambda'_1$ and $\Lambda'_2$ be arbitrary admissible lifts of $\bar\Lambda'$. Let $\varphi_1$, $\varphi_2$, $X_1$, $X_2$, $\Lambda_1$ and $\Lambda_2$ be constructed as above. Then $\Lambda_1$ and $\Lambda_2$ are both admissible lifts of $\bar \Lambda$ by Proposition \ref{prop_tilted_admissible}, and therefore they are both isomorphic to $\Lambda_0$ by assumption.
	For $i\in\{1,2\}$ let $\alpha_i:\ \Lambda_0 \stackrel{\sim}{\longrightarrow} \Lambda_i$ be an isomorphism. 
	Then ${^{\alpha_i}}X_i$ for $i\in \{1,2\}$ are two-sided tilting complexes in $\mathcal D^b(\Lambda_0^{\op} \otimes \Lambda_i')$. The restriction to the right of $k\otimes {{^{\alpha_i}}X_i}$ is of course still isomorphic to $T^{\varphi_i}$, and therefore we have the following isomorphisms in $\mathcal D^b(k\otimes \Lambda_0)$
	\begin{equation}
		\underbrace{(k\otimes {^{\alpha_1}X_1})\otimes_{k\otimes \Lambda_1'} {^{\varphi_1}\bar{\Lambda'}^{\varphi_2}}}_{\iso T^{\varphi_2} \iso {k\otimes {^{\alpha_2}X_2}}}
		\otimes_{k\otimes \Lambda_2'} (k\otimes {^{\alpha_2}X_2})^{\vee} \iso k\otimes \Lambda_0
	\end{equation}
	Of course this is only a (quasi-)isomorphism of complexes of left modules, but we can still deduce that 
	the left hand side has homology concentrated in a single degree. This implies that 
	$(k\otimes {^{\alpha_1}X_1})\otimes_{k\otimes \Lambda_1'} {^{\varphi_1}\bar{\Lambda}'^{\varphi_2}}\otimes_{k\otimes \Lambda_2'} (k\otimes {^{\alpha_2}X_2})^{\vee}$ is isomorphic to
	$(k\otimes \Lambda_0)^\beta$ for some automorphism $\beta$ of $k\otimes \Lambda_0$. Hence we get
	\begin{equation}\label{eqn_jksjksdd}
		{^{\varphi_1}\bar{\Lambda}'^{\varphi_2}} \iso (k\otimes {^{\alpha_1}X_1})^\vee\otimes_{k\otimes \Lambda_0}
		(k\otimes \Lambda_0)^\beta \otimes_{k\otimes \Lambda_0} (k\otimes {^{\alpha_2}X_2})
	\end{equation}
	which is now a quasi-isomorphism of two-sided complexes.
	By assumption there exists an automorphism $\gamma\in \Aut(\Lambda_0)$ which induces the same action on
	$K_0(k\otimes \Lambda_0)$ as $\beta$. In particular $(k\otimes {^{\alpha_1}X_1})^\vee\otimes_{k\otimes \Lambda_0}
	(k\otimes \Lambda_0)^\beta \iso (k\otimes {^{\alpha_1}X_1})^\vee\otimes_{k\otimes \Lambda_0}(k\otimes \Lambda_0^\gamma)$
	in $\mathcal D^b(\Lambda_0)$ (i. e. again forgetting about the left action), since $(k\otimes {^{\alpha_1}X_1})^\vee$ restricted to the right is a two-term tilting complex, and these are determined by their terms (and $\beta$ and $\id_k\otimes \gamma$ act on these terms in the same way by definition). This implies that if we replace 
	$(k\otimes \Lambda_0)^\beta$ by $k\otimes \Lambda_0^\gamma$ in the right hand side of \eqref{eqn_jksjksdd}, we still get a complex with 
	homology concentrated in a single degree. Moreover, since all involved complexes have terms which are projective 
	as both left and right modules, we have
	\begin{equation}
		(k\otimes {^{\alpha_1}X_1})^\vee\otimes_{k\otimes \Lambda_0}
		(k\otimes \Lambda_0^\gamma) \otimes_{k\otimes \Lambda_0} (k\otimes {^{\alpha_2}X_2}) \iso 
		k\otimes (({^{\alpha_1} X_1)^\vee \otimes_{\Lambda_0} \Lambda_0^\gamma \otimes_{\Lambda_0} {^{\alpha_2} X_2}})
	\end{equation}
	in $\mathcal D^b(\Lambda_1'^{\op}\otimes \Lambda_2')$. We conclude that $Y := ({^{\alpha_1} X_1)^\vee \otimes_{\Lambda_0} \Lambda_0^\gamma \otimes_{\Lambda_0} {^{\alpha_2} X_2}}$ is a two-sided tilting complex all of whose terms are projective from the left and from the right such that $k\otimes Y$ is isomophic to a stalk complex. By the properties of tilting complexes we revisited above we can conclude that $Y$ actually affords a Morita equivalence. As $\Lambda_1'$ and $\Lambda_2'$ are basic this implies $\Lambda_1'\iso \Lambda_2'$.
\end{proof}

\begin{proof}[Proof of Theorem \ref{thm_main}]
	First we need to show that (the basic order of) a block of quaternion defect of $\OO G$ with three simple modules
	is admissible in the sense of Definition \ref{defi_admissible}. It is clearly symmetric and it has split semisimple 
	$K$-span by Proposition \ref{prop_splitting}. That takes care of the first two properties required for admissibility.
	
	In \cite[Definition following Corollary 2.7]{Olsson}, the blocks of quaternion defect are divided into three different cases, labeled ``(aa)'', ``(ab)'' and ``(bb)''.  The blocks with three simple modules correspond to the case ``(aa)'' (see \cite[table on page 231]{Olsson}). By \cite[Theorem 1]{CabanesPicaronny} there is a perfect isometry between any two blocks of quaternion defect with the same label (i. e. either ``(aa)'', ``(ab)'', or ``(bb)''). In particular, there is a perfect isometry between any block of quaternion defect with three simple modules and the principal block of
	$\OO \SL_2(q)$ for an appropriately chosen $q$. A perfect isometry gives rise to an isomorphism between centers 
	and an isometry between Grothendieck groups satisfying the required properties of $\Phi$ and $\widehat{\Phi}$. This takes care of the third property required in the definition of admissibility. 
	It follows that  $\Lambda$ and $\Gamma$ are both admissible, and so are their basic orders. 
	
	By \cite[Chapter IX]{TameClass} there are only three possible basic algebras for $k\otimes \Lambda$ respectively 
	$k\otimes \Gamma$ for any fixed generalized quaternion defect group. These are the algebras $Q(3\mathcal A)^c_2$, $Q(3\mathcal B)^c$ and $Q(3\mathcal K)^c$. Hence the basic algebras of $k\otimes \Lambda$ and $k\otimes \Gamma$ each are isomorphic to one of those (for appropriate $c$). Any two admissible lifts of 
	$Q(3\mathcal K)^c$ are isomorphic by Lemma \ref{lemma_q3k}. The algebra $Q(3\mathcal A)_2^c$ is derived equivalent to $Q(3\mathcal K)^c$ by means of a two-term tilting complex, and Corollary \ref{corollary_full_permutation_auto}  implies that the second condition of Proposition \ref{prop_lift_transfer} is satisfied for the (unique) admissible lift
	of $Q(3\mathcal K)^c$, and hence Proposition \ref{prop_lift_transfer} yields that any two admissible lifts of $Q(3\mathcal A)_2^c$ are isomorphic. By comparing Cartan matrices one sees that $B_0(\OO \SL_2(q))$ for an appropriately chosen $q\equiv 1 \mod 4$ is a lift of $Q(3\mathcal A)_2^c$, and hence Corollary \ref{corollary_full_permutation_auto} and Proposition \ref{prop_lift_transfer} can be applied again. It follows that any two admissible lifts of $Q(3\mathcal B)^c$ are isomorphic. This completes the proof of the first assertion.
	
	The second assertion follows from the fact that the algebras $k\otimes\Lambda$ and $k\otimes \Gamma$ are derived equivalent, together with Theorem \ref{thm_rickard_lifting}, Proposition \ref{prop_tilted_admissible} and the uniqueness of admissible lifts we just showed.
\end{proof}

\paragraph{Acknowledgments} 
This research was supported by the Research Foundation Flanders (FWO - Vlaanderen) project G.0157.12N.

\bibliographystyle{plain}
\bibliography{refs}{}

\end{document}